\documentclass[12pt,reqno]{amsart}
\usepackage{amsmath}
\usepackage{amssymb}
\usepackage{verbatim}
\usepackage{mathrsfs}
\usepackage{graphicx} 
\usepackage{caption}
\usepackage{subcaption}
\usepackage{hyperref}

\usepackage[top=0.5in,bottom=0.5in,left=0.7in,right=0.7in]{geometry}

\newtheorem{lemma}{Lemma}
\newtheorem{prop}[lemma]{Proposition}

\newtheorem{theorem}[lemma]{Theorem}

\newtheorem{rem}[lemma]{Remark}

\newtheorem{definition}[lemma]{Definition}

\numberwithin{equation}{section}
\numberwithin{lemma}{section}

\newcommand{\C}{\mathbb{C}}    
\newcommand{\N}{\mathbb{N}}    
\newcommand{\R}{\mathbb{R}}    

\newcommand{\wh}{\widehat}
\renewcommand{\le}{\leqslant}
\renewcommand{\ge}{\geqslant}
\newcommand{\bs}{\backslash}
\newcommand{\ol}{\overline}
\newcommand{\la}{\langle}
\newcommand{\ra}{\rangle}
\newcommand{\bo}{\mathscr{O}} 

\newcommand{\eps}{\epsilon}

\newcommand{\ID}{{\text{Id}}}

\newcommand{\re}{\text{Re}}
\newcommand{\im}{\text{Im}}

\DeclareMathOperator{\DAS}{DAS}
\DeclareMathOperator{\AS}{AS}
\newcommand{\er}{\eqref}

\newcommand{\bp}{ \begin{proof} }
	\newcommand{\ep}{\hfill \end{proof} }
\newcommand{\be}{ \begin{equation} }
\newcommand{\ee}{ \end{equation} }
\newcommand{\tp}{\mathsf{T}}
\newcommand{\dm}{\mathsf{M}} 

\newcommand{\td}{\boldsymbol{\delta}}  

\newcommand{\sd}{\mathcal{S}}  
\newcommand{\tz}{\mathcal{T}}  

\newcommand{\cV}{\mathcal{V}}

\newcommand{\cW}{\mathcal{W}}

\newcommand{\dN}{\mathbb{N}^d}

\newcommand{\dR}{\mathbb{R}^d}
\newcommand{\dT}{\mathbb{T}^d}
\newcommand{\dZ}{\mathbb{Z}^d}

\newcommand{\dDC}{\mathscr{D}(\mathbb{R}^d)}  

\newcommand{\dLp}[1]{L_{#1}(\mathbb{R}^d)}
\newcommand{\dLr}[2]{(L_{#1}(\mathbb{R}^d))^{#2}}
\newcommand{\dLrs}[3]{(L_{#1}(\mathbb{R}^d))^{#2\times #3}}

\newcommand{\dTLp}[1]{L_{#1}(\mathbb{T}^d)}

\newcommand{\dlp}[1]{l_{#1}(\mathbb{Z}^d)}
\newcommand{\dlr}[2]{(l_{#1}(\mathbb{Z}^d))^{#2}}
\newcommand{\dlrs}[3]{(l_{#1}(\mathbb{Z}^d))^{#2\times #3}}
\newcommand{\dsq}{l(\mathbb{Z}^d)} 

\newcommand{\supp}{\text{supp}}

\newcommand{\tr}{\text{Trace}}
\newcommand{\vertiii}[1]{{\left\vert\kern-0.25ex\left\vert\kern-0.25ex\left\vert #1 
		\right\vert\kern-0.25ex\right\vert\kern-0.25ex\right\vert}}

\begin{document}
	
	\title[Mixed Dilation Framelets]{Framelets and Wavelets with Mixed Dilation Factors}

\author{Ran Lu}

\address{School of Mathematics,
	Hohai University, Nanjing,\quad Jiangsu, China 211106.
	\quad {\tt  rlu3@hhu.edu.cn}
}
\thanks{The research of the author was supported by the National Natural Science Foundation of China under grants 12201178 and 12271140.
}


\makeatletter \@addtoreset{equation}{section} \makeatother	

\begin{abstract} As a main research area in applied and computational harmonic analysis, the theory and applications of framelets have been extensively investigated. Most existing literature is devoted to framelet systems that only use one dilation matrix as the sampling factor. To keep some key properties such as directionality, a framelet system often has a high redundancy rate. To reduce redundancy, a one-dimensional tight framelet with mixed dilation factors has been introduced for image processing. Though such tight framelets offer good performance in practice, their theoretical properties are far from being well understood. In this paper, we will systematically investigate framelets with mixed dilation factors, with arbitrary multiplicity in arbitrary dimensions. We will first study the discrete framelet transform employing a filter bank with mixed dilation factors and discuss its various properties. Next, we will introduce the notion of a discrete affine system in $\dlp{2}$ and study discrete framelet transforms with mixed dilation factors. Finally, we will discuss framelets and wavelets with mixed dilation factors in the space $\dLp{2}$.
\end{abstract}


\maketitle

\section{Introduction}

In this paper, we study framelets and wavelets with mixed dilation factors. To explain our motivations, let us briefly recall some definitions of classical wavelet and framelet theory.

A $d\times d$ integer matrix $\dm$ is called a \emph{dilation matrix} if all its eigenvalues are greater than $1$ in modulus. By $f\in \dLrs{2}{r}{s}$ we mean that $f$ is an $r\times s$ matrix of $\dLp{2}$ functions. In particular, $(\dLp{2})^r:=\dLrs{2}{r}{1}$. For all $f\in \dLrs{2}{r}{s}$ and $g\in\dLrs{2}{t}{s}$, define
 $$\la f,g\ra_{\dLp{2}}:=\int_{\dR}f(x)\ol{g(x)}^{\tp}dx.$$ 
Let $\phi:=(\phi_1,\dots,\phi_r)^{\tp},\tilde{\phi}:=(\tilde{\phi}_1,\dots,\tilde{\phi}_r)^{\tp}\in(\dLp{2})^r$ and $\psi:=(\psi_1,\dots,\psi_s)^{\tp},\tilde{\psi}:=(\tilde{\psi}_1,\dots,\tilde{\psi}_s)^{\tp}\in(\dLp{2})^s$. Define an \emph{affine systems} via
\be\label{as}\AS(\{\phi;\psi\}):=\{\phi_\ell(\cdot-k): \ell=1,\dots,r;k\in\dZ\}\cup\{\psi_{\ell,\dm^j;k}:\ell=1,\dots,s; j\in\dN_0;k\in\dZ\},\ee
where $\psi_{\ell,\dm^j;k}:=|\det(\dm)|^{\frac{j}{2}}\psi_\ell(\dm^j\cdot-k)$. Similarly define $\AS(\{\tilde{\phi};\tilde{\psi}\})$. We say that $\{\phi;\psi\}$ is an \emph{$\dm$-framelet (or $\dm$-adic framelet)} in $\dLp{2}$ if $\AS(\{\phi;\psi\})$ is a frame in $\dLp{2}$, i.e., there exist positive constants $C_1$ and $C_2$ such that
\be\label{wav:fr}C_1\|f\|_{\dLp{2}}^2\le \sum_{\ell=1}^r\sum_{k\in\dZ}|\la f,\phi_\ell(\cdot-k)\ra_{\dLp{2}}|^2+\sum_{\ell=1}^s\sum_{j=0}^\infty\sum_{k\in\dZ}|\la f,\psi_{\ell,\dm^j;k}\ra_{\dLp{2}}|^2\le C_2\|f\|_{\dLp{2}}^2,\quad \forall f\in\dLp{2},\ee
where $|\la f, \psi_{\ell,\dm^j;k}\ra|^2:=\|\la f, \psi_{\ell,\dm^j;k}\ra\|_{l_2}^2$. The number $r$ is called the \emph{multiplicity} of the wavelet frame. $\{\phi;\psi\}$ is called a \emph{scalar framelet} if $r=1$, and a vector framelet or multiframelet if $r>1$. We will use the term framelet to refer to both unless specified. For $\phi,\tilde{\phi}\in(\dLp{2})^r$ and $\psi,\tilde{\psi}\in(\dLp{2})^s$, we say that $(\{\phi;\psi\},\{\tilde{\phi},\tilde{\psi}\})$ is a \emph{dual $\dm$-framelet} if $(\AS(\{\phi;\psi\},\AS(\{\tilde{\phi};\tilde{\psi}\})$ is a pair of dual frames in $\dLp{2}$, that is, both $\{\phi;\psi\}$ and $\{\tilde{\phi},\tilde{\psi}\}$ are $\dm$-wavelet frames in $\dLp{2}$ and satisfy
\be\label{dual:fr}f=\sum_{\ell=1}^r\sum_{k\in\dZ}\la f,\phi_\ell(\cdot-k)\ra_{\dLp{2}}\tilde{\phi}_\ell(\cdot-k)+\sum_{\ell=1}^s\sum_{j=0}^\infty\sum_{k\in\dZ}\la f,\psi_{\ell,\dm^j;k}\ra_{\dLp{2}}\tilde{\psi}_{\ell,\dm^j;k},\quad \forall f\in\dLp{2},\ee
with the above series converging unconditionally in $\dLp{2}$. We say that $(\{\phi;\psi\},\{\tilde{\phi},\tilde{\psi}\})$ is a \emph{biorthogonal $\dm$-wavelet} in $\dLp{2}$ if $(\AS(\{\phi;\psi\},\AS(\{\tilde{\phi};\tilde{\psi}\})$ is a biorthogonal basis in $\dLp{2}$. $\dm$-adic framelets/wavelets are often known as framelets/wavelets with uniform sampling/dilation factors.

By $\dlrs{0}{r}{s}$ we denote the linear space of all finitely supported matrix-valued sequences, i.e., $u\in \dlrs{0}{r}{s}$ if $u:\dZ\to\C^{r\times s}$ is a sequence such that $\{k\in\dZ: u(k)\neq 0\}$ is a finite set. For $a\in \dlrs{0}{r}{r}$, we define $\wh{a}(\xi):=\sum_{k\in\dZ}a(k)e^{-ik\cdot\xi}$. We say that $\phi\in(\dLp{2})^r$ is an \emph{$\dm$-refinable vector function} with a refinement mask/filter $a\in\dlrs{0}{r}{r}$ if 
\be\label{ref:phi}\phi(x)=\sum_{k\in\dZ}a(k)\phi(\dm x-k),\qquad x\in\dR,\ee
or equivalently 
\be\label{ref:phi:f}\wh{\phi}(\dm^{\tp}\xi)=\wh{a}(\xi)\wh{\phi}(\xi),\qquad\xi\in\dR.\ee
Here $\wh{f}(\xi):=\int_{\dR}f(x)e^{-ix\cdot\xi}$ is the Fourier transform of $f\in \dLp{1}$, which can be naturally extended to $\dLp{2}$ functions and tempered distributions. $\wh{\phi}$ is the vector obtained by taking entry-wise Fourier transform on $\phi$. A dual $\dm$-framelet $(\{\phi;\psi\},\{\tilde{\phi};\tilde{\psi}\})$ is often derived from a pair of refinable vector functions. Let $\phi,\tilde{\phi}\in(\dLp{2})^r$ be $\dm$-refinable vector functions with refinement filters $a,\tilde{a}\in\dlrs{0}{r}{r}$ respectively. As shown in \cite[Theorem 7.1.8]{hanbook}, under mild conditions of $a$ and $\tilde{a}$, one can construct a dual $\dm$-framelet $(\{\phi;\psi\},\{\tilde{\phi};\tilde{\psi}\})$ satisfying
\be\label{ref:filter}\wh{\psi}(\dm^{\tp}\xi)=\wh{b}(\xi)\wh{\psi}(\dm),\qquad\wh{\tilde{\psi}}(\dm^{\tp}\xi)=\wh{\tilde{b}}(\xi)\wh{\tilde{\psi}}(\dm),\ee
for a.e. $\xi\in\dR$, for some matrix-valued sequences/filters $b,\tilde{b}\in\dlrs{0}{s}{r}$ such that $(\{a;b\},\{\tilde{a};\tilde{b}\})$ forms a \emph{dual $\dm$-framelet filter bank}, i.e.,
\be\label{dual:fb}\ol{\wh{a}(\xi)}^{\tp}\wh{\tilde{a}}(\xi+2\pi\omega)+\ol{\wh{b}(\xi)}^{\tp}\wh{\tilde{b}}(\xi+2\pi\omega)=\td(\omega)I_r,\quad \omega\in\Omega_{\dm},\quad \xi\in\dR,\ee
where
$$\td(x)=\begin{cases}1, &x=0,\\
0, &x\in\dR\bs\{0\}.\end{cases}, $$
$I_r$ is the $r\times r$ identity matrix, and
$$\Omega_{\dm}:=(\dm^{-\tp}\dZ)\cap[0,1)^d.$$ 

Framelets generalize wavelets by allowing redundancy properties in their systems, which offers flexibility in their constructions. One major problem with multi-dimensional real wavelets is the lack of directionality, which significantly reduces the effectiveness of capturing edge singularities of the data in high dimensions. The added redundancy of framelets improves the shift-invariance property and increases the directionality, which greatly improves the effectiveness when handling multi-dimensional data (\cite{CD95,dhrs03,han97,han12,han14,hz14,hzz16,NGK99,NGK01,rs97,SBK05}). The degree of redundancy of a dual framelet is measured by its \emph{redundancy rate}. Here, let us roughly explain the redundancy rate of a wavelet/framelet transform. Let $\dm$ be a $d\times d$ dilation matrix, and let $a,\tilde{a}\in\dlrs{0}{r}{r}$ and $b,\tilde{b}\in\dlrs{0}{s}{r}$ be such that $(\{a;b\},\{\tilde{a};\tilde{b}\})$ is a dual $\dm$-framelet filter bank. A discrete framelet transform employing the filter bank $(\{a;b\},\{\tilde{a};\tilde{b}\})$ is implemented through convolution and sampling operations. For an input data $v\in(\dsq)^{1\times r}$, at the decomposition procedure of the transform, $v$ is convolved with $a^\star$ and $b^\star$ as $v*a^\star:=\sum_{k\in\dZ}v(k)a^\star(\cdot-k)$ and $v*b^\star$, where $u^\star:=\ol{u(-\cdot)}^{\tp}$ for any matrix-valued filter $u$, and then downsample as $w_1:=(v*a^\star)\downarrow \dm$ and $w_2:=(v*b^\star)\downarrow \dm$, where the down sampling operator is defined as in \er{dsamp}. $w_1$ and $w_2$ are called the framelet coefficients. A $J$-level discrete framelet transform implements the decomposition recursively $J$-times with $v$ being replaced by $(v*a^\star)\downarrow \dm$ as the new input data at each step. At the reconstruction procedure, the data $w_j$ is upsampled for $j=1,2$ as $w_j\uparrow\dm$, where the upsampling operator is defined as in \er{usamp}, and then is convolved with $\tilde{a}$ and $\tilde{b}$ as $(w_j\uparrow\dm)*\tilde{a}$ and $(w_j\uparrow\dm)*\tilde{b}$. By summing up all reconstructed data, we obtain a single output data. In many applications, we consider an initial data $v\in\dlrs{0}{1}{r}$ which has finite support and is real-valued, say $v$ contains $n$ real numbers. One first extends $v$ to a periodic sequence $v^{per}\in(\dsq)^{1\times r}$, and uses $v^{per}$ as the input data of the discrete framelet transform. Noting that $v^{per}*a^\star$ and $v^{per}*b^\star$ are both periodic sequences which have the same period as $v^{per}$, both sequences contain only finitely many real numbers, that is, $|\{\re((v^{per}*a^\star)(n))_l,\im((v^{per}*a^\star)(n))_l: n\in\dZ; l=1,\dots,r\}|:=n_a$ and $|\{\re((v^{per}*b^\star)(n))_l,\im((v^{per}*b^\star)(n))_l: n\in\dZ;l=1,\dots,r\}|:=n_b$ are finite numbers. Thus, the framelet coefficients $w_1$ and $w_2$ contain finitely many real numbers. If one implements a $J$-level discrete framelets transform, the framelet coefficients at the $J$-th level contain finitely many, say $N$ real numbers. The ratio $N/n$ is the redundancy rate of the $J$-level discrete framelet transform. For a $J$-level discrete framelet transform employing a dual $\dm$-framelet filter bank $(\{a;b\},\{\tilde{a};\tilde{b}\})$ with $a,\tilde{a}\in\dlrs{0}{r}{r}$ and $b,\tilde{b}\in\dlrs{0}{s}{r}$, its redundancy rate only depends on the level $J$, the number $s$ and the dilation factor $\dm$.

%

 Though the high redundancy rate of a dual framelet improves performance in practice, the computational cost will also increase as dimension increases. To our best knowledge, all currently known ($\dm$-adic) framelet systems with shift invariance property and directionality have various (high) degrees of redundancy. Recently, a version of tensor product complex tight framelets (TP-$\C$TF) was introduced in \cite{hzz16}, which offers good directionality and has a significantly lower redundancy rate compared with existing systems in the literature. This newly developed TP-$\C$TF is a particular case of \emph{mixed dilated framelet systems}. To be specific, let $\psi^0=(\psi^0_1,\dots,\psi^0_r)^{\tp}\in\dLr{2}{r},\psi^1,\dots,\psi^s\in \dLp{2}$ and $\dm_0,\dots,\dm_s$ be $d\times d$ dilation matrices. Define an affine system via
 \be\label{afs}\begin{aligned}\AS(\{\psi^l!\dm_l\}_{l=0}^s):=&\{\psi^0_q(\cdot-k):q=1,\dots,r; k\in\dZ\}\\
 	&\cup\{|\det(\dm_0^{-1}\dm_l)|^{\frac{1}{2}}\psi^l_{\dm_0^j;\dm_0^{-1}\dm_lk}:l=1,\dots,s;j\in\N_0;k\in\dZ\}.
 \end{aligned}\ee
 We say that $(\{\psi^l!\dm_l\}_{l=0}^s,\{\tilde{\psi}^l!\dm_l\}_{l=0}^s)$ is a \emph{dual framelet in $\dLp{2}$ with mixed dilation factors} if both $(\AS(\{\psi^l!\dm_l\}_{l=0}^s)$ and $\AS(\{\tilde{\psi}^l!\dm_l\}_{l=0}^s))$ are frames of $\dLp{2}$ and satisfy
 \be\label{dual:frt}\begin{aligned}\la f,g\ra_{\dLp{2}}=&\sum_{k\in\dZ}\la f,\psi^0(\cdot-k)\ra_{\dLp{2}}\la \tilde{\psi}^0(\cdot-k),g\ra_{\dLp{2}}\\
 	&+\sum_{l=1}^s\sum_{j=0}^\infty\sum_{k\in\dZ}|\det(\dm_0^{-1}\dm_l)|\la f,\psi^l_{\dm_0^j;\dm_0^{-1}\dm_lk}\ra_{\dLp{2}}\la \tilde{\psi}^l_{\dm_0^j;\dm_0^{-1}\dm_lk},g\ra_{\dLp{2}}\end{aligned}\ee
 for all $f,g\in\dLp{2}$, with the above series converging absolutely. We say that $(\{\psi^l!\dm_l\}_{l=0}^s,\{\tilde{\psi}^l!\dm_l\}_{l=0}^s)$ is a \emph{biorthogonal wavelet in $\dLp{2}$ with mixed dilation factors} if $(\AS(\{\psi^l!\dm_l\}_{l=0}^s),\AS(\{\tilde{\psi}^l!\dm_l\}_{l=0}^s))$ is a pair of biorthogonal bases for $\dLp{2}$. Unlike a traditional $\dm$-adic framelets system in which all generators are dilated by the same factor $\dm$, we now dilate the framelet generators by different dilation factors. Allowing mixed dilation factors creates flexibility in constructing framelets and makes achieving a low redundancy rate while keeping other desired properties possible. There is a small amount of work in the literature on lowering redundancy rates of wavelet/framelet systems by considering mixed dilation factors (e.g., \cite{SBBV03,WB12}).

 To our best knowledge, the only existing work in the literature discussing mixed dilated framelets theoretically is \cite{hzz16}, which only addresses tight framelets with multiplicity $r=1$. In this paper, we shall systematically develop the theory of framelets and wavelets with mixed dilation factors with arbitrary multiplicity. Our work will cover all classical $\dm$-adic framelets/wavelets as special cases. The structure of the paper is as follows: In Section~\ref{sec:dft:mix}, we introduce the notion of a discrete framelet transform with mixed dilation factors and study its various properties, including the perfect reconstruction property and the stability of filter banks. In Section~\ref{sec:wv:mix}, we introduce the notion of a wavelet filter bank with mixed dilation factors, which is a special case of framelet filter banks. We will provide a characterization of wavelet filter banks. We will further study discrete framelet/wavelet transforms through affine systems in Section~\ref{sec:fr:l2}. Finally, in Section~\ref{sec:wv:mix}, we shall connect discrete framelet/wavelet filter banks and framelets/wavelets in $\dLp{2}$. The work presented in this paper is an extension of the theory of tight framelet filter banks with mixed dilation factors in \cite{hzz16}, the new elements are the following: (1) the definition of a dual multiframelet with mixed dilation factors is a straight forward generalization of the definition of a scalar tight framelet with mixed dilation factors which was introduced in \cite{hzz16}. Thus, there is no surprise that several corresponding properties of a scalar tight framelets transfer to a dual multiframelet, with only minor modification; (2) we define a wavelet filter bank with mixed dilation factors by following the way we do in the uniform sampling case, but the study is more difficult under the setting of mixed dilation factors. We will discuss this in detail in Section~\ref{sec:wv:mix}.

\section{Discrete Framelet Transforms with Mixed Dilation Factors}
\label{sec:dft:mix}
In this section, we provide a self-contained systematic study of discrete framelet transforms with mixed dilation factors. 

By $(\dsq)^{r\times s}$ we denote the linear space of all sequences $u:\dZ\to\C^{r\times s}$. For a dilation matrix $\dm$ and a matrix filter, we define:
\begin{itemize}
	\item The \emph{subdivision operator:}
	\be\label{sd}[\sd_{a,\dm}u](n):=|\det(\dm)|^{\frac{1}{2}}\sum_{k\in\dZ}u(k)a(n-\dm k),\quad n\in\dZ,\quad a\in\dlrs{0}{r}{q},\quad u\in(\dsq)^{s\times r}.\ee
	
	\item The \emph{transition operator:}
	\be\label{tz} [\tz_{b,\dm}u](n):=|\det(\dm)|^{\frac{1}{2}}\sum_{k\in\dZ}u(k)\ol{b(k-\dm n)}^{\tp},\quad n\in\dZ,\quad b\in\dlrs{0}{t}{r},\quad  u\in(\dsq)^{s\times r}.\ee
\end{itemize}

For any filter $u\in(\dsq)^{s\times r}$, define the following notations:

\begin{itemize}
	\item \emph{The down sampling sequence with respect to $\dm$}:
	\be\label{dsamp} [u\downarrow \dm](k)=u(\dm k),\qquad k\in\dZ.\ee
	
	\item \emph{The up sampling sequence with respect to $\dm$}: 
	\be\label{usamp}[u\uparrow \dm](k)=\begin{cases}
		u(\dm^{-1}k), &k\in \dm\dZ\\[0.5cm]
		0, &\text{else}
	\end{cases}.\ee
	
	\item The filter $u^{\star}$ is defined via $u^{\star}(k):=\ol{u(-k)}^{\tp}$ for all $k\in\dZ$.
	
	\item The \emph{(formal) Fourier series} of $u$:
	\be\label{mf:fs}\wh{u}(\xi):=\sum_{k\in\dZ}u(k)e^{-ik\cdot\xi},\qquad \xi\in\dR.\ee
\end{itemize}

The convolution of two filters $u_1$ and $u_2$ is defined via
\be\label{conv:fil}(u_1*u_2)(n):=\sum_{k\in\dZ}u_1(k)u_2(n-k).\ee
The subdivision and transition operators can be represented by
\be\label{sd:tz:samp}\sd_{a,\dm}u=|\det(\dm)|^{\frac{1}{2}}(u\uparrow\dm)*a,\qquad \tz_{b,\dm}u=|\det(\dm)|^{\frac{1}{2}}(u*b^{\star})\downarrow\dm,\ee

By letting $\Omega_{\dm}:=[\dm^{-\tp}\dZ]\cap[0,1)^d$, we have
\be\label{sd:ft}\wh{\sd_{a,\dm}u}(\xi)=|\det(\dm)|^{\frac{1}{2}}\wh{u}(\dm^{\tp}\xi)\wh{a}(\xi),\ee
\be\label{tz:ft}\wh{\tz_{b,\dm}u}(\xi)=|\det(\dm)|^{-\frac{1}{2}}\sum_{\omega\in\Omega_{\dm}}\wh{u}(\dm^{-\tp}\xi+2\pi\omega)\ol{\wh{b}(\dm^{-\tp}\xi+2\pi\omega)}^{\tp}.\ee

For $u\in\dlrs{2}{t}{r}$ and $v\in\dlrs{2}{s}{r}$, we define:
\be\label{l2:inprod}\la u,v\ra_{\dlp{2}}:=\sum_{k\in\dZ}u(k)\ol{v(k)}^{\tp}.\ee
Note that $\la u,v\ra_{\dlp{2}}\in\C^{t\times s}$ is a matrix of complex numbers.\\

We now discuss some important properties of the subdivision and transition operators. The following two lemmas are the matrix-valued filter versions of \cite[Lemma 2.3 and Lemma 4.3]{han13}.

\begin{lemma}\label{sdtr} Let $a\in\dlrs{0}{r}{q}$ be a finitely supported matrix-valued filter and $\dm$ be a $d\times d$ dilation matrix. Then, the following operators are well-defined:
	$$\sd_{a,\dm}:\dlrs{2}{s}{r}\to\dlrs{2}{s}{q},\quad \tz_{a,\dm}:\dlrs{2}{s}{q}\to\dlrs{2}{s}{r}.$$
	Moreover, we have $\sd_{a,\dm}=\tz_{a,\dm}^{\star}$. That is
	\be\label{sd:tz:adj}\la \sd_{a,\dm}v,w\ra_{\dlp{2}}=\la v,\tz_{a,\dm}w\ra_{\dlp{2}},\qquad  v\in\dlrs{2}{s}{r},w\in\dlrs{2}{s}{q}.\ee
\end{lemma}

\begin{lemma}\label{comp}Let $u_1\in\dlrs{0}{q}{t},u_2\in\dlrs{0}{r}{q}$ be finitely supported matrix-valued filters and let $\dm_1,\dm_2$ be $d\times d$ dilation matrices. Then
	\be\label{sd:tz:2}\sd_{u_1,\dm_1}\sd_{u_2,\dm_2}v=\sd_{(u_2\uparrow \dm_1)*u_1,\dm_1\dm_2}v,\qquad \tz_{u_2,\dm_2}\tz_{u_1,\dm_1}w=\tz_{(u_2\uparrow \dm_1)*u_1,\dm_1\dm_2}w,\ee
	for all $v\in(\dsq)^{s\times r}$ and $w\in(\dsq)^{s \times t}$.
\end{lemma}

\subsection{Multi-level Discrete Framelet Transform and the Perfect Reconstruction Property}
Suppose we have finitely supported matrix-valued filters $b_0,\tilde{b}_0\in\dlrs{0}{r}{r}, b_1,\dots,b_s,\tilde{b}_1,\dots,\tilde{b}_s\in\dlrs{0}{1}{r}$, and $d\times d$ dilation matrices $\dm_0,\dm_1,\dots,\dm_s$. We now state the discrete framelet transform employing the filter bank $(\{b_l!\dm_l\}_{l=0}^s,\{\tilde{b}_l!\dm_l\}_{l=0}^s)$ with mixed dilation factors. For $J\in\N$, the $J$-level discrete framelet transform can be stated as the following:

\begin{enumerate}
	\item[Step 1.] For any input signal $v_{0,0}\in(\dsq)^{1\times r}$, perform the $J$-level discrete framelet decomposition procedure:  \be\label{dft:decomp}  v_{l,j}=\tz_{b_l,\dm_l}v_{0,j-1},\qquad j=1,\dots,J,\quad l=0,\dots,s.\ee

	\item[Step 2.] Set $\tilde{v}_{l,J}:=v_{l,J}$ for all $l=0,\dots,s$. Recursively compute $\tilde{v}_{j-1}, j=J,\dots,1$ via
	\be\label{dft:recons}\tilde{v}_{j-1}:=\sum_{l=0}^s\sd_{\tilde{b}_l,\dm_l}\tilde{v}_{l,j},\qquad j=J,\dots,1.\ee
	
\end{enumerate}

We say that a $J$-level discrete framelet transform has \emph{the perfect reconstruction (PR) property} if the original input data $v_{0}$ is equal to the output data $\tilde{v}_{0}$. 
To study the $J$-level discrete framelet transform employing the filter bank 
$(\{b_l!\dm_l\}_{l=0}^s,\{\tilde{b}_l!\dm_l\}_{l=0}^s)$,
we need to define the following linear operators:
\begin{enumerate}
	\item[(1)] \emph{The $J$-level discrete framelet analysis operator} employing the filter bank  $\{b_l!\dm_l\}_{l=0}^s$:
	\be\label{dist:anal}\cW_J:(\dsq)^{1\times r}\to (\dsq)^{1\times (sJ+r)},\quad \cW_Jv:=(v_{1,1},\dots,v_{s,1},\dots,v_{1,J},\dots,v_{s,J},v_{0,J}),\qquad\forall v\in\dsq, \ee
	where $v_{l,j},l=0,\dots,s, j=1,\dots,J$ are defined as \er{dft:decomp}. Denote $\cW:=\cW_1$.

	\item[(2)] 	 \emph{The $J$-level discrete framelet synthesis operator} employing the filter bank $\{\tilde{b_1}!\dm_l\}_{l=0}^s$:
	\be\label{dist:synt}\tilde{\cV}_J:(\dsq)^{1\times(sJ+r)}\to(\dsq)^{1\times r},\quad \tilde{\cV_J}(\tilde{v}_{1,1},\dots,\tilde{v}_{s,1},\dots,\tilde{v}_{1,J},\dots,\tilde{v}_{s,J},\tilde{v}_{0,J}):=\tilde{v}_0, \ee
	for all $\tilde{v}_{1,1},\dots,\tilde{v}_{s,1},\dots,\tilde{v}_{1,J},\dots,\tilde{v}_{s,J}\in \dsq,\tilde{v}_{0,J}\in(\dsq)^{1\times r}$, where $\tilde{v}_{0}$ is calculated via \er{dft:recons}. Denote $\tilde{\cV}:=\tilde{\cV}_1$.
\end{enumerate} 
It follows immediately that the $J$-level discrete framelet transform has the PR property if and only if 
\be\label{J:pr}\tilde{\cV}_J\cW_J=\ID_{(\dsq)^{1\times r}}.\ee
Moreover, by noting that
\be\label{cW:dft}\cW_J=(\ID_{(\dsq)^{1\times [s(J-1)]}}\otimes\cW )\dots( \ID_{(\dsq)^{1\times s}}\otimes\cW)\cW,\ee
and
\be\label{cV:dft}\tilde{\cV}_J=\tilde{\cV}(\ID_{(\dsq)^{1\times s}}\otimes\tilde{\cV} )\dots(\ID_{(\dsq)^{1\times [s(J-1)]}}\otimes \tilde{\cV}),\ee
it is obvious that \er{J:pr} is equivalent to
\be\label{1:pr}\tilde{\cV}\cW=\ID_{(\dsq)^{1\times r}}.\ee

Here, we characterize the PR property of a multi-level discrete framelet transform with mixed dilation factors.

\begin{theorem}\label{pr:mix}Let $b_0,\tilde{b}_0\in\dlrs{0}{r}{r}, b_1,\dots,b_s,\tilde{b}_1,\dots,\tilde{b}_s\in\dlrs{0}{1}{r}$ be finitely supported matrix-valued filters and $\dm,\dots,\dm_s$ be $d\times d$ dilation matrices. The following statements are equivalent:
	
	\begin{enumerate}
		\item[(i)]For every $J\in\N$, the $J$-level discrete framelet transform employing the filter bank\\* $(\{b_l!\dm_l\}_{l=0}^s,\{\tilde{b}_l!\dm_l\}_{l=0}^s)$ has the PR property.
		
		\item [(ii)] $\tilde{\cV}\cW v=v$ for all $v\in\dlrs{0}{1}{r}$, where $\cW$ and $\tilde{\cV}$ are the $1$-level discrete analysis and synthesis operators respectively.
		
		\item [(iii)]$(\{b_l!\dm_l\}_{l=0}^s,\{\tilde{b}_l!\dm_l\}_{l=0}^s)$ is a dual framelet filter bank with mixed dilation factors, i.e.,
		\be\label{pr:frt:mix}\sum_{l=0}^s\chi_{\Omega_l}(\omega)\ol{\wh{b_l}(\xi)}^{\tp}\wh{\tilde{b}_l
		}(\xi+2\pi\omega)=\td(\omega)I_r,\ee
		for almost every $\xi\in\dR$ and every $\omega\in\bigcup_{l=0}^s\Omega_l$, where 
		\be\label{omega:l} \Omega_l:=[\dm_l^{-\tp}\dZ]\cap[0,1)^d,\qquad l=0,\dots,s,\ee 
		and $\chi_E$ is the indicator function of the subset $E\subseteq\dR$.
	\end{enumerate}
	
\end{theorem}

\bp (i) $\Leftrightarrow$ (ii): (i) $\Rightarrow$ (ii) is trivial. To prove (ii) $\Rightarrow$ (i), one uses the locality of the subdivision and transition operators (see the proofs of \cite[Theorem 1.1.1]{hanbook} and \cite[Theorem 2.1]{han13}).\\

(i) $\Rightarrow$ (iii): 
By definition, we have
\be\label{vw:1}\tilde{\cV}\cW v=\sum_{l=0}^s\sd_{\tilde{b}_l,\dm_l}\tz_{b_l,\dm_l}v,\qquad v\in(\dsq)^{1\times r}.\ee
By item (i) and taking Fourier series on both sides of \er{vw:1}, we have
\be\label{pr:fourier}\wh{v}(\xi)=\sum_{l=0}^s\sum_{\omega_l\in\Omega_l}\wh{v}(\xi+2\pi\omega_l)\ol{\wh{b_l}(\xi+2\pi\omega_l)}^{\tp}\wh{\tilde{b}_l}(\xi),\qquad \xi\in\dR.\ee
For simplicity of presentation, define $\Omega:=\bigcup_{l=0}^s\Omega_l$. For every $\omega \in\Omega$, define
\be\label{u:omega}\wh{u_\omega}(\xi):=\left(\sum_{l=0}^s\chi_{\Omega_l}(\omega)\ol{\wh{b_l}(\xi+2\pi\omega)}^{\tp}\wh{\tilde{b}_l}(\xi)\right)-\td(\omega),\qquad \xi\in\dR.\ee
It follows that \er{pr:fourier} is equivalent to
\be\label{pr:f2}\wh{v}(\xi)=\sum_{\omega\in\Omega}\wh{v}(\xi+2\pi\omega)\wh{u_\omega}(\xi),\qquad\xi\in\dR.\ee
Define
\be\label{eps:omega}\eps_1:=\inf_{x\in\Omega\setminus\{0\}}\|x\|_\infty,\quad \eps_2:=\inf_{x\in\Omega\setminus\{0\}, y\in\{0,1\}^d}\|x-y\|_\infty.\ee
Let $x_0\in(-\pi,\pi)^d$ be fixed. For every $\eps\in(0,\min(\eps_1,\eps_2))$, define $E_{x_0,\eps}:=x_0+[-\frac{\eps}{\pi},\frac{\eps}{\pi}]^d$. Let $v\in(\dsq)^{1\times r}$ be such that
\be\label{supp:v}(\supp(\wh{v})\cap[-\pi,\pi)^d)\subseteq E_{x_0,\eps}\subseteq (-\pi,\pi)^d,\ee
where $\supp(\wh{v})$ is the support of the function $\wh{v}$. Note that for any $\xi\in E_{x_0,\eps}$ and $\omega\in\Omega\setminus\{0\}\subseteq[0,1)^d$, the point $\xi+2\pi\omega$ lies in $[-\pi,\pi)^d+2\pi y$ for some (unique) $y\in\{0,1\}^d$. Moreover, for any $\xi\in E_{x_0,\eps}$, $\omega\in\Omega\setminus\{0\}$ and  $y\in\{0,1\}^d$, we have
\be\label{eps:delta}\|\xi+2\pi\omega-x_0\|_\infty\ge 2\pi\|\omega\|_\infty-\|\xi-x_0\|_\infty>2\pi\eps-\frac{\eps}{\pi}>\frac{\eps}{\pi},\ee
and 
\be\label{eps:delta:y}\|\xi+2\pi\omega-(x_0+2\pi y)\|_\infty\ge 2\pi\|\omega-y\|_\infty-\|\xi-x_0\|_\infty>2\pi\eps-\frac{\eps}{\pi}>\frac{\eps}{\pi}.\ee
It follows that 
\be\label{whv:omega}\wh{v}(\xi+2\pi\omega)=0,\qquad\forall \xi\in E_{x_0,\eps},\quad\omega\in\Omega\setminus\{0\}.\ee
Thus
\be\label{whv:0}\wh{v}(\xi)\wh{u_0}(\xi)=0,\qquad\forall \xi\in E_{x_0,\eps}.\ee
Since $x_0\in(-\pi,\pi)^d$ is arbitrary and $E_{x_0,\eps}$ is an open neighbourhood of $x_0$, we conclude that $\wh{u_0}(\xi)=0$ for all $\xi\in(-\pi,\pi)^d$. Since $\wh{u_0}$ is $2\pi\mathbb{Z}^d$-periodic, it follows that $\wh{u_0}=0$ for a.e. $\xi\in\mathbb{R}^d$. For $\mathring{\omega}\in\Omega\setminus\{0\}$, note that \eqref{pr:f2} is equivalent to
\begin{equation}\label{vwv3}\sum_{\omega\in \Omega}\wh{v}(\xi+2\pi(\omega-\mathring{\omega}))\wh{u_\omega}(\xi-2\pi\mathring{\omega})=0,\quad\forall\xi\in\mathbb{R}^d.
\end{equation}
By applying the same argument as above, we have $\wh{u_{\mathring{\omega}}}=0$ for a.e. $\xi\in\dR$. This proves item (iii).\\

(iii) $\Rightarrow$ (i): This follows immediately from item (iii) and the fact that \er{pr:f2} holds for all $v\in\dlrs{0}{1}{r}$.

\ep

\begin{rem}\label{dft:msamp}It is not hard to observe that $(\{b_l!\dm_l\}_{l=0}^s,\{\tilde{b}_l!\dm_l\}_{l=0}^s)$ is a dual framelet filter bank if and only if $(\{\tilde{b}_l!\dm_l\}_{l=0}^s,\{b_l!\dm_l\}_{l=0}^s)$ is a dual framelet filter bank.
\end{rem}

\subsection{The Stability of Framelet Filter Banks}
We now discuss the stability of a discrete framelet transform with mixed dilation factors.	

\begin{definition}A filter bank $\{b_l!\dm_l\}_{l=0}^s$ \textbf{has stability} in $\dlp{2}$ if there exist $C_1,C_2>0$ such that
\be\label{l2:sta}C_1\|v\|_{\dlrs{2}{1}{r}}^2\le \|\cW_Jv\|_{\dlrs{2}{1}{[(sJ+r)]}}^2\le C_2\|v\|_{\dlrs{2}{1}{r}}^2\ee
	for all $J\in\N$ and $v\in\dlrs{2}{1}{r}$, where $\cW_J$ is the $J$-level discrete analysis operator employing the filter bank $\{b_l!\dm_l\}_{l=0}^s$. In this case, $\{b_l!\dm_l\}_{l=0}^s$ is called \textbf{a framelet filter bank} with mixed dilation factors.
\end{definition}

\begin{theorem}\label{stfr}Let $(\{b_l!\dm_l\}_{l=0}^s,\{\tilde{b}_l!\dm_l\}_{l=0}^s)$ be a dual framelet filter bank with mixed dilation factors, and let $0<C_1\leq C_2<\infty$. For each $J\in\N$, let $\cW_J$ and $\cV_J$ (resp. $\tilde{\cW}_J$ and $\tilde{\cV}_J$) be the $J$-level discrete framelet analysis and synthesis operators employing the filter bank $\{b_l!\dm_l\}_{l=0}^s$ (resp. $\{\tilde{b}_l!\dm_l\}_{l=0}^s$). The following statements are equivalent.
	\begin{enumerate}
		\item $\{b_l!\dm_l\}_{l=0}^s$ and $\{\tilde{b}_l!\dm_l\}_{l=0}^s$ have stability in $\dlp{2}$ with \er{l2:sta} and
	\be\label{l2:sta:t}C_2^{-1}\|v\|_{\dlrs{2}{1}{r}}^2\le\|\tilde{W}_Jv\|_{\dlrs{2}{1}{[(sJ+r)]}}^2\le C_1^{-1}\|v\|_{\dlrs{2}{1}{r}}^2\ee
		hold for all $J\in\N$ and $v\in\dlrs{2}{1}{r}$.
		
		\item $\|\cW_J\|^2\le C_2$ and $\|\tilde{\cW}_J\|^2\le C_1^{-1}$ for all $J\in\N$.
		
		\item $\|\cV_J\|^2\le C_2$ and $\|\tilde{\cV}\|^2\le C_1^{-1}$ for all $J\in\N$.
		
		\item $\|\cV_J\|^2\le C_2$ and $\|\tilde{\cW}_J\|^2\le C_1^{-1}$ for all $J\in\N$.
		
		\item $\|\cW_J\|^2\le C_2$ and $\|\tilde{\cV}_J\|^2\le C_1^{-1}$ for all $J\in\N$.
		
	\end{enumerate}
\end{theorem}

\bp By Lemma~\ref{sdtr}, we have
\be\label{V:W:adj}\la \cV_Jw,v\ra_{\dlp{2}}=\la w,\cW_Jv\ra_{\dlp{2}},\quad \la \tilde{\cV}_Jw,v\ra_{\dlp{2}}=\la w,\tilde{\cW}_Jv\ra_{\dlp{2}}\ee
for all $J\in\N$, $w\in\dlrs{2}{1}{[(sJ+r)]}$ and $v\in\dlrs{2}{1}{r}$. Thus we have (2)$\Leftrightarrow$(3)$\Leftrightarrow$(4)$\Leftrightarrow$(5). We finish the proof by proving (1)$\Leftrightarrow$(2). It's trivial to see that (1)$\Rightarrow$(2). To prove the converse, note that by Remark~\ref{dft:msamp}, we have $\cV_J\tilde{\cW}_Jv=\tilde{\cV}_J\cW_Jv=v$ for all $v\in\dlp{2}$ and $J\in\N$. Moreover, by Lemma\autoref{sdtr}, we have $\|\cW_J\|=\|\cV_J\|$ and $\|\tilde{\cW}_J\|=\|\tilde{\cV}_J\|$ for all $J\in\N$. Thus item (2) yields
\be\label{VW:op}\|\cV_J\tilde{\cW}_Jv\|_{\dlrs{2}{1}{r}}^2=\|v\|_{\dlp{2}}^2\le C_2\|\tilde{\cW}_Jv\|_{\dlrs{2}{1}{[(sJ+r)]}}^2,\ee
\be\label{VW:op:1}\|\tilde{\cV}_J\cW_Jv\|_{\dlrs{2}{1}{r}}^2=\|v\|_{\dlrs{2}{1}{r}}^2\le C_1^{-1}\|\cW_Jv\|_{\dlrs{2}{1}{[(sJ+r)]}}^2,\ee
for all $J\in\N$ and $v\in\dlrs{2}{1}{r}$. Thus 
\be\label{VW:op:2}\|\cW_J\|^2=\|\cV_J\|^2\le C_2,\quad \|\tilde{\cW}_J\|^2=\|\tilde{\cV}_J\|^2\le C_1^{-1}.\ee
This proves (2)$\Rightarrow$(1), and the proof is now complete.
\ep

\section{Wavelet Filter Banks with Mixed Dilation Factors}

In this section, we introduce the notion of a wavelet filter bank with mixed dilation factors. We generalize the definition of an $\dm$-adic biorthogonal wavelet filter bank, which uses a uniform dilation factor $\dm$. To do this, we need the following proposition on the analysis and synthesis operators.

\label{sec:wv:mix}
\begin{prop}\label{wvdef}Let $b_0,\tilde{b}_0\in\dlrs{0}{r}{r}, b_1,\dots,b_s,\tilde{b}_1,\dots,\tilde{b}_s\in\dlrs{0}{1}{r}$ be finitely supported filters and $\dm_0,\dots,\dm_s$ be $d\times d$ dilation matrices such that $(\{b_l!\dm_l\}_{l=0}^s,\{\tilde{b}_l!\dm_l\}_{l=0}^s)$ is a dual framelet filter bank with mixed sampling factors. Let $\cW$ and $\tilde{\cV}$ be the discrete framelet analysis and synthesis operators employing the filter banks $\{b_l!\dm_l\}_{l=0}^s$ and $\{\tilde{b}_1!\dm_l\}_{l=0}^s$. The following statements are equivalent.
	
	\begin{enumerate}
		\item[(i)] $\cW$ is surjective.
		
		\item[(ii)] $\tilde{\cV}$ is injective.
		
		\item[(iii)] $\tilde{\cV}\cW=\ID_{(\dsq)^{1\times r}}$ and $\cW\tilde{\cV}=\ID_{(\dsq)^{1\times (s+r)}}$.
	\end{enumerate}
\end{prop}

\begin{proof}(iii) $\Rightarrow$ (i) and (iii) $\Rightarrow$ (ii) are trivial. 
	
	We now prove that either (i) or (ii) implies (iii). The condition $\tilde{\cV}\cW=\ID_{(\dsq)^{1\times r}}$ follows immediately from Theorem\autoref{pr:mix}. It remains to show that either (i) or (ii) implies that $\cW\tilde{\cV}=\ID_{(\dsq)^{1\times [r(s+1)]}}$. First, assume item (i) holds. For every $v\in(\dsq)^{1\times r}$ we have
	\be\label{wav:fil}\cW v=\cW(\tilde{\cV}\cW)v=\cW\tilde{\cV}(\cW v).\ee
	By the surjectivity of $\cW$, we conclude that $\cW\tilde{\cV}=\ID_{(\dsq)^{1\times [r(s+1)]}}$. Now assume item (ii) holds. For every $w\in (\dsq)^{1\times (s+r)}$, we have
	\be\label{wav:fil:1}\tilde{\cV}(\cW\tilde{\cV})w=(\tilde{\cV}\cW)\tilde{\cV}w=\tilde{\cV}w.\ee
	Since $\tilde{\cV}$ is injective, we see that $(\cW\tilde{\cV})w=w$ for all $w\in  (\dsq)^{1\times (s+r)}$. 
	
\end{proof}

\begin{definition} A dual framelet filter bank $(\{b_l!\dm_l\}_{l=0}^s,\{\tilde{b}_l!\dm_l\}_{l=0}^s)$ is called a biorthogonal wavelet filter bank with mixed dilation factors if any one of the equivalent conditions in Proposition~\ref{wvdef} holds.\end{definition}

The following theorem characterizes the biorthogonal wavelet filter banks with mixed sampling factors. 

\begin{theorem}Let $b_0,\tilde{b}_0\in\dlrs{0}{r}{r}, b_1,\dots,b_s,\tilde{b}_1,\dots,\tilde{b}_s\in\dlrs{0}{1}{r}$ be finitely supported filters and let $\dm,\dots,\dm_s$ be $d\times d$ dilation matrices. The following are equivalent:
	
	\begin{enumerate}
		\item[(i)] $(\{b_l!\dm_l\}_{l=0}^s,\{\tilde{b}_l!\dm_l\}_{l=0}^s)$ is a biothogonal wavelet filter bank with mixed dilation factors.
		
		\item[(ii)] $(\{b_l!\dm_l\}_{l=0}^s,\{\tilde{b}_l!\dm_l\}_{l=0}^s)$ is a dual framelet filter bank with mixed dilation factors, and satisfies
		\be\label{bi:wfb}\sum_{\omega\in\Omega_k\cap\Omega_l}\wh{b_k}(\xi+2\pi \omega)\ol{\wh{\tilde{b}_l}(\xi+2\pi\omega)}^{\tp}=\begin{cases}I_r, &k=l=0\\
			\td(k-l), &k,l\neq 0\\
			0, &k\neq 0, l=0\\
			0, &k=0, l\neq0,\end{cases}\qquad \text{a.e. }\xi\in\dR,\qquad k,l=0,\dots,s,\ee
		where $\Omega_l$ is defined via \er{omega:l}.	
	\end{enumerate}
	
\end{theorem}

\bp Let $\cW$ and $\tilde{\cV}$ be the discrete framelet analysis and synthesis operators employing the filter bank $(\{b_l!\dm_l\}_{l=0}^s,\{\tilde{b}_l!\dm_l\}_{l=0}^s)$.\\

(i) $\Rightarrow$ (ii): If item (i) holds, then by item (iii) of Proposition~\ref{wvdef}, we have $\tilde{\cV}\cW=\ID_{\dsq}$. This means that $(\{b_l!\dm_l\}_{l=0}^s,\{\tilde{b}_l!\dm_l\}_{l=0}^s)$ must be a dual framelet filter bank with mixed dilation factors. \\

On the other hand, let $w:=(w_1,\dots,w_s,w_0)$ with $w_1,\dots,w_s\in\dlp{1}$ and $w_0\in\dlrs{1}{1}{r}$. By $\cW\tilde{\cV}=\ID_{(\dsq)^{1\times (s+r)}}$, we have
\be\label{WV:id}w_k=(\cW\tilde{\cV}w)_k=\sum_{l=0}^s\tz_{b_k,\dm_k}\sd_{\tilde{b}_l,\dm_l}w_l,\qquad k=0,\dots,s.\ee
By taking the Fourier series of both sides of \er{WV:id}, we have
\be\label{f:WV:id}\begin{aligned}
	\wh{w_k}(\xi)&=\sum_{l=0}^s|\det(\dm_k)|^{-\frac{1}{2}}\sum_{\omega\in\Omega_k}\wh{\sd_{\tilde{b_l},\dm_l}w_l}(\dm_k^{-\tp}\xi+2\pi\omega)\ol{\wh{b_k}(\dm_k^{-\tp}+2\pi\omega)}^{\tp}\\
	&=\sum_{l=0}^s|\det(\dm_k)|^{-\frac{1}{2}}|\det(\dm_l)|^{\frac{1}{2}}\sum_{\omega\in\Omega_k}\wh{w_l}(\dm_l^{\tp}(\dm_k^{-\tp}\xi+2\pi\omega))\wh{\tilde{b_l}}(\dm_k^{-\tp}\xi+2\pi\omega)\ol{\wh{b_k}(\dm_k^{-\tp}+2\pi\omega)}^{\tp}\\
	&=\sum_{l=0}^s\sum_{\omega\in\Omega_k}\wh{w_l}(\dm_l^{\tp}(\dm_k^{-\tp}\xi+2\pi\omega))\wh{u_{l,k}}(\dm_k^{-\tp}\xi+2\pi\omega)\\
\end{aligned}\ee
where
\be\label{u:k:l}\wh{u_{l,k}}(\xi):=|\det(\dm_k)|^{-\frac{1}{2}}|\det(\dm_l)|^{\frac{1}{2}}\wh{\tilde{b_l}}(\xi)\ol{\wh{b_k}(\xi)}^{\tp},\qquad \xi\in\dR.\ee
As $(\dm_l^{-\tp}\dZ)\cap(\dm_k^{-\tp}\dZ)$ is a sublattice of $\dm_k^{-\tp}\dZ$ which contains $\dZ$, there exist  $\omega_{k,1},\dots,\omega_{k,L_{k,l}}\in\Omega_k$ with $\omega_{k,1}=0$ and $L_{k,l}:=\frac{\#\Omega_k}{\#(\Omega_k\cap\Omega_l)}$ such that
\be\label{sub:lat}\dm_k^{-\tp}\dZ=\bigsqcup_{j_{k,l}=1}^{L_{k,l}}\omega_{k,j_{k,l}}+[(\dm_l^{-\tp}\dZ)\cap(\dm_k^{-\tp}\dZ)].\ee
It follows that
\be\label{sub:lat:2}\begin{aligned}
	&\sum_{\omega\in\Omega_k}\wh{w_l}(\dm_l^{\tp}(\dm_k^{-\tp}\xi+2\pi\omega))\wh{u_{l,k}}(\dm_k^{-\tp}\xi+2\pi\omega)\\
	=&\sum_{j_{k,l}=1}^{L_{k,l}}\wh{w_l}(\dm_l^{\tp}\dm_k^{-\tp}\xi+2\pi\dm_l^{\tp}\omega_{k,j_{k,l}}))\sum_{\omega\in\Omega_k\cap\Omega_l}\wh{u_{l,k}}(\dm_k^{-\tp}\xi+2\pi\dm_l^{\tp}\omega_{k,j_{k,l}}+2\pi\omega).
\end{aligned}\ee
Consequently, we have
\be\label{f:WV:id:2}\wh{w_k}(\xi)=\sum_{l=0}^s\sum_{j_{k,l}=1}^{L_{k,l}}\wh{w_l}(\dm_l^{\tp}\dm_k^{-\tp}\xi+2\pi\dm_l^{\tp}\omega_{k,j_{k,l}}))\sum_{\omega\in\Omega_k\cap\Omega_l}\wh{u_{l,k}}(\dm_k^{-\tp}\xi+2\pi\dm_l^{\tp}\omega_{k,j_{k,l}}+2\pi\omega),\qquad\xi\in\dR,\ee
for all $w\in(\dsq)^{1\times(s+1)}$. Denote $\{e^r_1,\dots,e^r_r\}$ the standard basis of $\mathbb{R}^r$. Choose $w_0:=\pmb{\delta}e^r_j, j=1,\dots,r$ and $w_1=\dots=w_s:=0$, we conclude that the special case of \er{bi:wfb} with $k=l=0$ holds. Similarly choose $w_0:=0$, $w_k:=1$ for some $k\in\{1,\dots,s\}$ and $w_l:=0$ for all $l\neq k$, we can prove the special case of \er{bi:wfb} with $k=l$ and $k,l\neq 0$.\\

Next, we consider the case when $k\neq l$. First fix $k\in\{1,\dots,s\}$. Choose any $w=(w_1,\dots,w_s,w_0)\in\dlrs{1}{1}{(s+r)}$ such that $w_1=\dots=w_s=0$ and $w_0\neq 0$. By \er{f:WV:id:2},
\be\label{f:WV:id:4}0=\wh{w_k}(\dm_k^{\tp}\xi)=\sum_{j_{k,0}=1}^{L_{k,0}}\wh{w_0}(\dm_0^{\tp}\xi+2\pi\dm_0^{\tp}\omega_{k,j_{k,0}}))\sum_{\omega\in\Omega_0\cap\Omega_k}\wh{u_{0,k}}(\xi+2\pi\omega_{k,j_{k,0}}+2\pi\omega),\ee
for all $\xi\in\dR$. Note that if $\Omega_k\subseteq\Omega_0$, then $\dm_0^\tp\omega_{k,j_{k,0}}\in\dZ$ and $L_{k,0}=1$. Thus \er{f:WV:id:4} reduces to
\be\label{f:WV:id:5}0=\wh{w_0}(\dm_0^{\tp}\xi)\sum_{\omega\in \Omega_k}\wh{u_{0,k}}(\xi+2\pi\omega),\qquad \xi\in\dR.\ee
As $w_0\in \dlrs{1}{1}{r}\setminus\{0\}$ is arbitrary, it  follows immediately from \er{f:WV:id:5} that
\be\label{f:WV:id:8}\sum_{\omega\in \Omega_0\cap\Omega_k }\wh{u_{0,k}}(\xi+2\pi\omega)=0\ee
for all $\xi\in\dR$. If $\Omega_k\nsubseteq\Omega_0$, then $\dm_0^{\tp}\omega\notin\dZ$ for all $\omega\in\Omega_k\setminus\Omega_0$. Thus
\be\label{d:k:l}\delta:=\inf_{\omega\in \Omega_k\setminus\Omega_0, y\in\dZ}2\pi\|\dm_0^{\tp}\omega-y\|>0.\ee
Fix $x_0\in(-\pi,\pi)^d$ and set $y_0:=\dm_l^{\tp}x_0$. Choose $w_0$ such that $\eps\in(0,\delta/2)$ and $\wh{w_0}\in (C^\infty(\dT))^{1\times r}$ such that the following hold:
\begin{enumerate}
	\item $\wh{w_0}(\xi)=(e_j^r)^{\tp}$ for all $\xi\in B_\eps(y_0)$ with $j=1,\dots,r$, where
	$$B_\eps(y_0):=\{y\in\dR:\, \|y-y_0\|<\eps\};$$
	
	\item $\wh{w_0}(\xi)=0$ for all $\xi\in B_\eps(y_0)+2\pi\dm_0^{\tp}\omega$ with $\omega\in\Omega_k\setminus\Omega_0$;
	
	\item $\dm_0^{-\tp}B_\eps(y_0)\subseteq (-\pi,\pi)^d$;
	
	\item $\supp(\wh{w_0})\cap[y_0+(-\pi,\pi^d)]\subseteq B_{2\eps}(y_0)$.
	
\end{enumerate}
By our choice of $w_l$, it follows from \er{f:WV:id:5} that
\be\label{f:WV:id:6}\sum_{\omega\in\Omega_0\cap\Omega_k}\wh{u_{0,k}}(\xi+2\pi\omega)=0,\qquad\xi\in \dm_{0}^{\tp}B_{\eps}(y_0).\ee
As $x_0\in(-\pi,\pi)^d$ is arbitrary and $\wh{u_{0,k}}$ is $2\pi\dZ$-periodic, we conclude that \er{f:WV:id:8} holds for a.e. $\xi\in\dR$. Similarly, one can prove that
\be\label{f:WV:id:9}\sum_{\omega\in \Omega_l\cap\Omega_0 }\wh{u_{l,0}}(\xi+2\pi\omega)=0,\qquad \sum_{\omega\in \Omega_l\cap\Omega_k }\wh{u_{l,k}}(\xi+2\pi\omega)=\td(k-l),\ee
for a.e. $\xi\in\dR$ and for all $k,l=1,\dots,s$. This proves item (ii).\\

(ii) $\Rightarrow$ (i): Conversely, suppose item (ii) holds. Then by Theorem~\ref{pr:mix}, $\tilde{\cV}\cW=\ID_{(\dsq)^{1\times r}}$.\\

It remains to show that $\cW\tilde{\cV}=\ID_{(\dsq)^{1\times(s+r)}}$. Let $u_{k,l}$ be defined via \er{u:k:l}. For all $w=(w_1,\dots,w_s,w_0)\in\dlrs{0}{1}{[r(s+1)]}$, we have
\be\label{f:WV:id:7}\begin{aligned}\wh{(\cW\tilde{\cV}w)_k}(\xi)&=\sum_{l=0}^s\sum_{\omega\in\Omega_k}\wh{w_l}(\dm_l^{\tp}(\dm_k^{-\tp}\xi+2\pi\omega))\wh{u_{l,k}}(\dm_k^{-\tp}\xi+2\pi\omega)\\
	&=\sum_{j_{k,l}=1}^{L_{k,l}}\wh{w_l}(\dm_l^{\tp}\dm_k^{-\tp}\xi+2\pi\dm_l^{\tp}\omega_{k,j_{k,l}}))\sum_{\omega\in\Omega_l\cap\Omega_k}\wh{u_{l,k}}(\dm_k^{-\tp}\xi+2\pi\omega_{k,j_{k,l}}+2\pi\omega)\\
	&=\wh{w_k}(\xi),
\end{aligned}\ee
for a.e. $\xi\in\dR$ and all $k=0,1,\dots,s$, where the last equality follows from \er{bi:wfb}. Thus $\cW\tilde{\cV}w=w$ for all $w\in\dlrs{0}{1}{(s+r)}$. By using the locality of the subdivision and transition operators, we conclude that $\cW\tilde{\cV}=\ID_{(\dsq)^{1\times(s+r)}}$. This proves item (i).
\ep

The following result states the critical sampling property of a biorthogonal wavelet filter bank.

\begin{lemma}\label{bfb:csp}Let $b_0,\tilde{b}_0\in\dlrs{0}{r}{r}, b_1,\dots,b_s,\tilde{b}_1,\dots,\tilde{b}_s\in\dlrs{0}{1}{r}$ be finitely supported filters and let $\dm,\dots,\dm_s$ be $d\times d$ dilation matrices.
	If $(\{b_l!\dm_l\}_{l=0}^s,\{\tilde{b}_l!\dm_l\}_{l=0}^s)$ is a biorthogonal wavelet filter bank with mixed dilation factors, then
	\be\label{cr:samp}\frac{r}{|\det(\dm_0)|}+\sum_{l=1}^{s}\frac{1}{|\det(\dm_l)|}=r.\ee
\end{lemma}

\bp Denote $\cV$ and $\cW$ (resp. $\tilde{\cV}$ and $\tilde{\cW}$) the discrete framelet synthesis and analysis operators employing the filter bank $\{b_l!\dm_l\}_{l=0}^s$ (resp. $\{\tilde{b}_l!\dm_l\}_{l=0}^s$). By definition of a biorthogonal wavelet filter bank, we have $\cW\tilde{\cV}=\ID_{(\dsq)^{1\times(s+r)}}$. Fix $k=\in\{1,\dots,s\}$ and define $v:=(v_1,\dots,v_s,v_0)\in(\dsq)^{1\times(s+r)}$ via
\be\label{v:sp}v_0:=0,\qquad v_k:=\td,\qquad v_l=0,\qquad l\in\{1,\dots,s\}\setminus\{k\}.\ee	
By Lemma~\ref{sdtr} and \er{sd:ft}, we have
\be\label{V:W:adj:0}\begin{aligned}&1=\la v,v\ra_{\dlp{2}}=\la v,\cW\tilde{\cV}v\ra_{\dlp{2}}=\la \td,\tz_{b_k,\dm_k}\sd_{\tilde{b}_k,\dm_k}\td\ra_{\dlp{2}}\\
	=&\la \sd_{b_k,\dm_k}\td,\sd_{\tilde{b}_k,\dm_k}\td\ra_{\dlp{2}}=\la \wh{\sd_{b_k,\dm_k}}\td,\wh{\sd_{\tilde{b}_k,\dm_k}\td}\ra_{\dTLp{2}}\\
	=&(2\pi)^{-d}|\det(\dm_k)|\int_{[0,2\pi)^{d}}\wh{b_k}(\xi)\ol{\wh{\tilde{b}_k}(\xi)}^{\tp}d\xi\\
\end{aligned}\ee
Similarly, denote $\{e_1^r,\dots,e_r^r\}$ the standard basis of $\mathbb{R}^r$ and choose $w^j:=(w_1,\dots,w_s,w_0^j)\in(\dsq)^{1\times(s+r)}$ with
\be\label{w:sp}w_0^j:=(e_j^r)^{\tp}\td,\qquad w_l=0,\qquad l=1,\dots,s.\ee
One can conclude that
\be\label{V:W:adj:1}1=\la w^j,w^j\ra_{\dlp{2}}=(2\pi)^{-d}|\det(\dm_0)|\int_{[0,2\pi)^{d}}\left(\wh{b_0}(\xi)\ol{\wh{\tilde{b}_0}(\xi)}^{\tp}\right)_{j,j}d\xi,\qquad j=1,\dots,r.\ee
It follows from \er{pr:frt:mix}, \er{V:W:adj:0} and \er{V:W:adj:1} that
\be\label{tr:pr}\begin{aligned}
	r=\tr(I_r)&=\sum_{l=0}^s(2\pi)^{-d}\int_{[0,2\pi)^d}\tr\left(\ol{\wh{b}_l(\xi)}^{\tp}\wh{\tilde{b}_l}(\xi)\right)d\xi\\
	&=\sum_{l=0}^s(2\pi)^{-d}\int_{[0,2\pi)^d}\tr\left(\ol{\wh{b}_l(\xi)}\wh{\tilde{b}_l}(\xi)^{\tp}\right)d\xi\\
	&=\frac{r}{|\det(\dm_0)|}+\sum_{l=1}^s\frac{1}{|\det(\dm_l)|}.
\end{aligned}\ee

\ep

\section{Discrete Affine Systems in $\dlp{2}$}
\label{sec:das}
In this section, we further study the discrete framelet transforms by introducing the notion of a discrete affine system in $\dlp{2}$.\\

Let $b_0,\tilde{b}_0\in\dlrs{0}{r}{r}, b_1,\dots,b_s,\tilde{b}_1,\dots,\tilde{b}_s\in\dlrs{0}{1}{r}$ be finitely supported filters, and let $\dm_0,\dots,\dm_s$ be $d\times d$ dilation matrices. Define
\be\label{blj} \wh{b_{l,j}}(\xi):=\wh{b_l}((\dm_0^{\tp})^{j-1}\xi)\wh{b_0}((\dm_0^{\tp})^{j-2}\xi)\dots\wh{b_0}(\dm_0^{\tp}\xi)\wh{b_0}(\xi),\ee
\be\label{tblj} \wh{\tilde{b}_{l,j}}(\xi):=\wh{\tilde{b}_l}((\dm_0^{\tp})^{j-1}\xi)\wh{\tilde{b}_0}((\dm^{\tp})^{j-2}\xi)\dots\wh{\tilde{b}_0}(\dm_0^{\tp}\xi)\wh{\tilde{b}_0}(\xi),\ee
for all $j\in\N$, $l=0,\dots,s$ and $\xi\in\dR$, with the convention $b_{0,0}:=\pmb{\delta}_0I_r=:\tilde{b}_{0,0}$, $b_{l,1}:=b_l$ and $\tilde{b}_{l,1}:=\tilde{b}_l$ for $l=1,\dots,s$. In other words:
\be\label{blj:1} b_{l,j}:=(b_l\uparrow \dm_0^{j-1})*(b_0\uparrow \dm_0^{j-2})*\dots*(b_0\uparrow \dm_0)*b_0,\ee
\be\label{tblj:1} \tilde{b}_{l,j}:=(\tilde{b}_l\uparrow \dm_0^{j-1})*(\tilde{b}_0\uparrow \dm_0^{j-2})*\dots*(\tilde{b}_0\uparrow \dm_0)*\tilde{b}_0.\ee
For all $k\in\dZ$, $J\in\N$ and $l=0,\dots,s$, define
\be\label{b:l:j:k}b_{l,j;k}:=|\det(\dm_0)|^{\frac{j-1}{2}}|\det(\dm_l)|^{\frac{1}{2}}b_{l,j}(\cdot-\dm_0^{j-1}\dm_lk),\ee
\be\label{tb:l:j:k}\tilde{b}_{l,j;k}:=|\det(\dm_0)|^{\frac{j-1}{2}}|\det(\dm_l)|^{\frac{1}{2}}\tilde{b}_{l,j}(\cdot-\dm_0^{j-1}\dm_lk).\ee

The following lemma is an important result of the multi-level discrete analysis and synthesis processes.

\begin{lemma}\label{key1}Let $b_0\in\dlrs{0}{r}{r},b_1,\dots,b_s\in\dlrs{0}{1}{r}$ be finitely supported filters, and let $\dm_0,\dots,\dm_s$ be $d\times d$ dilation matrices.	
	\begin{enumerate}	
		
		\item[(i)] For any fixed $v_{0,0}\in(\dsq)^{1\times r}$, define $v_{l,j}$ as in \er{dft:decomp} for all $l=0,\dots,s $ and $j=1,\dots,J$. Then we have
		\be\label{v:j:j:k}v_{l,j}(k)=\la v_{0,0},b_{l,j;k}\ra_{\dlp{2}}=\sum_{n\in\dZ}v_{0,0}(n)\ol{b_{l,j;k}(n)}^{\tp},\qquad k\in\dZ.\ee

		\item[(ii)]	For $v_{0,J},v_{l,j}\in\dsq $ with $l=1,\dots,s$ and $j=1,\dots,J$, we have
		\be\label{cv:J}\cV_J(0,\dots,0,v_{0,J})=\sum_{k\in\dZ}v_{0,J}(k)b_{0,J;k}\ee
		and
		\be\label{cv:J:1}\cV_J(0,\dots,0,v_{l,j},0,\dots,0)=\sum_{k\in\dZ}v_{l,j}(k)b_{l,j;k},\qquad l=1,\dots,s,\quad j=1,\dots,J,\ee
		where $\cV_J$ is the $J$-level discrete synthesis operator employing the filter bank $\{b_l!\dm_l\}_{l=0}^s$.	
		
	\end{enumerate}
\end{lemma}

\bp For every $k\in\dZ$, we have
$$\begin{aligned}v_{l,j}(k)&=\tz_{b_l,\dm_l}v_{0,j-1}(k)=\tz_{b_l,M_l}\mathcal{T}_{b_0,\dm_0}^{j-1}v_{0,0}(k)\\
&=\tz_{(b_l\uparrow \dm_0^{j-1})*(b_0\uparrow \dm_0^{j-2})*\dots*(b_0\uparrow \dm_0)*b_0,\dm_0^{j-1}\dm_l}v_{0,0}(k)\quad\text{(by Lemma~\ref{comp})}\\
&=\tz_{b_{l,j},\dm_0^{j-1}\dm_l}v_{0,0}(k)\\
&=|\det(\dm_l)|^{\frac{1}{2}}|\det(\dm_0)|^{\frac{j-1}{2}}\sum_{n\in\dZ}v_{0,0}(n)\ol{b_{l,j}(n-\dm_0^{j-1}\dm_lk)}^{\tp}\\
&=\langle v_{0,0},b_{l,j;k}\rangle_{\dlp{2}}.
\end{aligned}$$
This proves item (i).\\
To prove item (ii), Define
\be\label{dst:syn}\mathring{v}_{0,j-1}=\sd_{b_0,\dm_0}\mathring{v}_{l,j},\qquad j=J,\dots,1,\ee
with then convention $\mathring{v}_{0,J}=v_{0,J}$. Then
$$\begin{aligned}\cV_J(0,\dots,0,v_{0,J})&=\mathring{v}_{0,0}=\sd_{b_0,\dm_0}^Jv_{0,J}\\
&=\sd_{(b_0\uparrow \dm_0^{j-1})*\dots*(b_0\uparrow \dm_0)*b_0,\dm_0^J}v_{0,J}\quad\text{(by Lemma~\ref{comp})}\\
&=\sd_{b_{0,J},\dm_0^J}v_{0,J}\\
&=|\det(\dm_0)|^{\frac{J}{2}}\sum_{k\in\dZ}v_{0,J}(k)b_{0,J}(\cdot-\dm_0^Jk)\\
&=\sum_{k\in\dZ}v_{0,J}(k)b_{0,J;k}.
\end{aligned}$$
Similarly, one can conclude that
\be\label{cv:j:2}\cV_J(0,\dots,0,v_{l,j},0,\dots,0)=\sum_{k\in\dZ}v_{l,j}(k)b_{l,j;k},\qquad  l=1,\dots,s,\quad j=1,\dots,J.\ee
This proves item (ii).	
\ep

We now state the definition of a discrete affine system associated with a filter bank.

\begin{definition}Let $b_0\in\dlrs{0}{r}{r},b_1,\dots,b_s\in\dlrs{0}{1}{r}$ be finitely supported filters and let $\dm_0,\dots,\dm_s$ be $d\times d$ dilation matrices. For every $J\in\N$, the \textbf{$J$-level discrete affine system} associated to the filter bank $\{b_l!\dm_l\}_{l=0}^s$ is defined via
	\be\label{das}\DAS_J(\{b_l!\dm_l\}_{l=0}^s)=\{b_{0,J;k}:k\in\dZ\}\cup\{b_{l,j;k}: l=1,\dots,s; j=1,\dots,J;k\in\dZ\},\ee
	where $b_{l,j;k}$ is defined via \er{b:l:j:k} for all $k\in\dZ, j\in\N$ and $l=0,\dots,s$.
\end{definition}

The stability of a filter bank is naturally linked to the frame property of its associated discrete affine systems.

\begin{lemma}\label{stadas}Let $b_0\in\dlrs{0}{r}{r},b_1,\dots,b_s\in\dlrs{0}{1}{r}$ be finitely supported filters and let $\dm_0,\dots,\dm_s$ be $d\times d$ dilation matrices. Then the filter bank $\{b_l!\dm_l\}_{l=0}^s$ has stability in $\dlp{2}$ if and only if there exist $C_1,C_2>0$ such that
	\be\label{das:frm}C_1\|v\|_{\dlrs{r}{1}{r}}^2\le\sum_{k\in\dZ}\|\la v,b_{0,J;k}\ra_{\dlp{2}}\|^2+\sum_{j=1}^J\sum_{l=1}^s\sum_{k\in\dZ}|\la v,b_{l,j;k}\ra_{\dlp{2}}|^2\le C_2\|v\|_{\dlrs{2}{1}{r}}^2\ee
	holds for all $v\in\dlrs{2}{1}{r}$ and $J\in\N$.
\end{lemma}

\bp Let $\cW_J$ be the $J$-level discrete analysis operator employing the filter bank $\{b_l!\dm_l\}_{l=0}^s$. By item (i) of Lemma~\ref{key1}, we have
\be\label{sta:fb}\begin{aligned}\|\cW_Jv\|_{\dlrs{2}{1}{[r(sJ+1)]}}^2&=\|v_{0,J}\|_{\dlrs{2}{1}{r}}^2+\sum_{j=1}^J\sum_{l=1}^s\|v_{l,j}\|_{\dlrs{2}{1}{r}}^2\\
	&=\sum_{k\in\dZ}\|\la v,b_{0,J;k}\ra_{\dlp{2}}\|^2+\sum_{j=1}^J\sum_{l=1}^s\sum_{k\in\dZ}\|\la v,b_{l,j;k}\ra_{\dlp{2}}\|^2.
\end{aligned}\ee
Hence, the result follows immediately.

\ep

The associated discrete affine systems of a dual framelet filter bank have the frame expansion property, which is illustrated by the following result.

\begin{lemma}\label{fraff}Let $b_0,\tilde{b}_0\in\dlrs{0}{r}{r},b_1,\dots,b_s,\tilde{b}_1,\dots,\tilde{b}_s\in\dlrs{0}{1}{r}$ be finitely supported filters and let $\dm_0,\dots,\dm_s$  be $d\times d$ dilation matrices. Then $(\{b_l!\dm_l\}_{l=0}^s,\{\tilde{b}_l!\dm_l\}_{l=0}^s)$ is a dual framelet filter bank with mixed dilation factors if and only if
	\be\label{frdecomp}v=\sum_{k\in\dZ}\la v,b_{0,J;k}\ra_{\dlp{2}} \tilde{b}_{0,J;k}+\sum_{j=1}^J\sum_{l=1}^s\sum_{k\in\dZ}\la v,b_{l,j;k}\ra_{\dlp{2}} \tilde{b}_{l,j;k}\ee
	for all $v\in \dlrs{2}{1}{r}$ and $J\in\N$, where $b_{l,j;k}$ and $\tilde{b}_{l,j;k}$ are defined as \er{b:l:j:k} and \er{tb:l:j:k} for all $k\in\dZ, j\in\N$ and $l=0,\dots,s$.
\end{lemma}

\bp Let $v_{0,0}\in\dlrs{2}{1}{r}$ and $J\in\N$. Define $v_{l,j}$ as in \er{dft:decomp} for all $j=1,\dots,J$ and $l=0,\dots,s$. Then by Lemma~\ref{key1}, we have
\be\label{dft:key1}\begin{aligned}\tilde{\cV}_J\cW_Jv_{0,0}&=\tilde{\cV}_J(v_{1,1},\dots, v_{s,1},\dots,v_{1,J},\dots,v_{s,J},v_{0,J})\\
	&=\sum_{k\in\dZ}v_{0,J}(k) \tilde{b}_{0,J;k}+\sum_{j=1}^J\sum_{l=1}^s\sum_{k\in\dZ}\tilde{v}_{l,j}(k) \tilde{b}_{l,j;k}\\
	&=\sum_{k\in\dZ}\la v_{0,0},b_{0,J;k}\ra_{\dlp{2}} \tilde{b}_{0,J;k}+\sum_{j=1}^J\sum_{l=1}^s\sum_{k\in\dZ}\langle v_{0,0},b_{l,j;k}\rangle_{\dlp{2}} \tilde{b}_{l,j;k}.
\end{aligned}\ee
Hence $(\{b_l!\dm_l\}_{l=0}^s,\{\tilde{b}_l!\dm_l\}_{l=0}^s)$ is a dual framelet filter bank with mixed sampling factors if and only if \er{frdecomp} holds for all $v\in \dlrs{2}{1}{r}$ and $J\in\N$.
\ep

Now, we are at the stage to give a complete characterization of a dual framelet filter bank by using its associated discrete affine systems, which is summarized as the following theorem:

\begin{theorem}\label{frdas}Let $b_0,\tilde{b}_0\in\dlrs{0}{r}{r},b_1,\dots,b_s,\tilde{b}_1,\dots,\tilde{b}_s\in\dlrs{0}{1}{r}$ be finitely supported filters and let $\dm_0,\dots,\dm_s$ be $d\times d$ dilation matrices. The following statements are equivalent.
	\begin{enumerate}
		\item[(i)] $(\{b_l!\dm_l\}_{l=0}^s,\{\tilde{b}_l!\dm_l\}_{l=0}^s)$ is a dual framelet filter bank with mixed dilation factors.
		
		\item[(ii)] For all $v,w\in\dlrs{2}{1}{r}$:
		\be\label{dfrexp1}\la v,w\ra_{\dlp{2}}=\sum_{l=0}^s\sum_{k\in\dZ}\la v,b_{l,1;k}\ra_{\dlp{2}}\la \tilde{b}_{l,1;k},w\ra_{\dlp{2}}.\ee
		
		\item[(iii)] For all $J\in\N$ and $v,w\in\dlrs{2}{1}{r}$:
		\be\label{dfrexpJ}\la v,w\ra_{\dlp{2}}=\sum_{k\in\dZ}\langle v,b_{0,J;k}\ra_{\dlp{2}}\la \tilde{b}_{0,J;k},w\ra_{\dlp{2}}+\sum_{j=1}^J\sum_{l=1}^s\sum_{k\in\dZ}\la v,b_{l,j;k}\ra_{\dlp{2}}\la \tilde{b}_{l,j;k},w\ra_{\dlp{2}}.\ee

		\item[(iv)] (\textbf{Cascade structure}) For all $j\in\N$ and $v,w\in\dlrs{2}{1}{r}$:
		\be\label{dfrcas}\begin{aligned}&\sum_{k\in\dZ}\la v,b_{0,j-1;k}\ra_{\dlp{2}}\la \tilde{b}_{0,j-1;k},w\ra_{\dlp{2}}
			=\sum_{l=0}^s\sum_{k\in\dZ}\la v,b_{l,j;k}\ra_{\dlp{2}}\la \tilde{b}_{l,j;k},w\ra_{\dlp{2}}\end{aligned},\ee
		with the convention $\tilde{b}_{0,0}=b_{0,0}:=\td I_r$ and $\tilde{b}_{0,0;k}=b_{0,0;k}:=\td_kI_r$ where $\td_k:=\td(\cdot-k)$.\\

		***************************************************************\\
		\item[] If further assume that $\{b_l!M_l\}_{l=0}^s$ and $\{\tilde{b}_l!M_l\}_{l=0}^s$ have stability in $\dlp{2}$, then each of the above statements is equivalent to the following statement:
		
		\item[(v)] \er{frdecomp} holds, and moreover, there exist $C_1,C_2>0$ such that \er{das:frm} and
		\be\label{t:das:frm}C_2^{-1}\|v\|_{\dlrs{2}{1}{r}}^2\le\sum_{k\in\dZ}\|\la v,\tilde{b}_{0,J;k}\ra_{\dlp{2}}\|^2+\sum_{j=1}^J\sum_{l=1}^s\sum_{k\in\dZ}|\la v,\tilde{b}_{l,j;k}\ra_{\dlp{2}}|^2\le C_1^{-1}\|v\|_{\dlrs{2}{1}{r}}^2\ee
		hold for all $J\in\N$ and $v\in\dlrs{2}{1}{r}$.
		
	\end{enumerate}
\end{theorem}

\bp
(i)$\Leftrightarrow$(ii): Follows immediately from Lemma~\ref{fraff}.\\

(ii)$\Leftrightarrow$(iv): Suppose that (iv) holds. Then we have
\be\label{casc:0}\begin{aligned}v&=\sum_{k\in\dZ}\la v,\pmb{\delta}_kI_r\ra_{\dlp{2}}\pmb{\delta}_kI_r=\sum_{k\dZ}\la v,b_{0,0;k}\ra_{\dlp{2}}\tilde{b}_{0,0;k}=\sum_{l=0}^s\sum_{k\in\dZ}\la v,b_{l,1;k}\ra_{\dlp{2}}\tilde{b}_{l,1;k}
\end{aligned}\ee
for all $v\in\dlrs{2}{1}{r}$. Thus, (ii) holds. Conversely, suppose that (ii) holds. Note that for all $l=0,\dots,s$ and $j\in\N$, we have
\be\label{casc:1}b_{l,j}=(b_{l,1}\uparrow \dm_0^{j-1})*b_{0,j-1}=\sum_{n\in\dZ}(b_{l,1}\uparrow \dm_0^{j-1})(n)b_{0,j-1}(\cdot-n)=\sum_{k\in\dZ}b_{l,1}(k)b_{0,j-1}(\cdot-\dm_0^{j-1}k).\ee
For all $k\in\dZ$, $l=0,\dots,s$ and $j\in\N$, it follows that
\be\label{bljk}\begin{aligned}b_{l,j;k}&=|\det(\dm_0)|^{\frac{j-1}{2}}|\det(\dm_l)|^{\frac{1}{2}}b_{l,j}(\cdot-\dm_0^{j-1}\dm_lk)\\
	&=|\det(\dm_0)|^{\frac{j-1}{2}}|\det(\dm_l)|^{\frac{1}{2}}\sum_{m\in\dZ}b_{l,1}(m)b_{0,j-1}(\cdot-\dm_0^{j-1}\dm_lk-\dm_0^{j-1}m)\\
	&=|\det(\dm_0)|^{\frac{j-1}{2}}|\det(\dm_l)|^{\frac{1}{2}}\sum_{m\in\dZ}b_{l,1}(m-\dm_lk)b_{0,j-1}(\cdot-\dm_0^{j-1}m)\\
	&=\sum_{m\in\dZ}b_{l,1;k}(m)b_{0,j-1;m}.
\end{aligned}\ee
Thus for all $v\in\dlrs{2}{1}{r}$, $k\in\dZ$, $l=0,\dots,s$ and $j\in\N$, we have
\be\label{vbljk}\la v,b_{l,j;k}\ra_{\dlp{2}}=\sum_{m\in\dZ}\la v,b_{0,j-1;m}\ra_{\dlp{2}}\ol{b_{l,1;k}(m)}^{\tp}=\la B_{0,j-1;v},b_{l,1;k}\ra_{\dlp{2}},\ee
where
\be\label{casc:2}B_{0,j-1;v}(m)=\la v,b_{0,j-1;m}\ra_{\dlp{2}},\qquad m\in\dZ.\ee
It follows from \eqref{bljk} and\eqref{vbljk} that
\be\label{casc:4}\begin{aligned}
	\sum_{l=0}^s\sum_{k\in\dZ}\la v,b_{l,j;k}\ra_{\dlp{2}}\tilde{b}_{l,j;k}=&\sum_{l=0}^s\sum_{k\in\dZ}\la B_{0,j-1;v},b_{l,1;k}\ra_{\dlp{2}}\sum_{m\in\dZ}\tilde{b}_{l,1;k}(m)\tilde{b}_{0,j-1;m}\\
	=&\sum_{m\in\dZ}\left(\sum_{l=0}^s\sum_{k\in\dZ}\la B_{0,j-1;v},b_{l,1;k}\ra_{\dlp{2}}\tilde{b}_{l,1;k}(m)\right)\tilde{b}_{0,j-1;m}\\
	=&\sum_{m\in\dZ}B_{0,j-1;v}\tilde{b}_{0,j-1;m}\\
	=&\sum_{m\in\dZ}\la v,b_{0,j-1;m}\ra_{\dlp{2}}\tilde{b}_{0,j-1;m}.
\end{aligned}\ee
This proves (iv).\\

(ii)$\Leftrightarrow$(iii): (iii) trivially implies (ii). Conversely, (ii) and (iv) are equivalent and together imply (iii).\\

If we further assume that both $\{b_l!\dm_l\}_{l=0}^s$ and $\{\tilde{b}_l!\dm_l\}_{l=0}^s$ have stability in $\dlp{2}$, then by Lemma~\ref{stadas}, each of items (i)-(iv) is equivalent to (v).

\ep

Similarly, we have the following characterization of biorthogonal filter banks using discrete affine systems.

\begin{theorem}\label{wvdas}Let $b_0,\tilde{b}_0\in\dlrs{0}{r}{r}$ and $b_1,\dots,b_s,\tilde{b}_1,\dots,\tilde{b}_s\in\dlrs{0}{1}{r}$ be finitely supported filters and let $\dm_0,\dots,\dm_s$ be $d\times d$ dilation matrices. The following statements are equivalent.
	\begin{enumerate}
		\item[(i)] $(\{b_l!\dm_l\}_{l=0}^s,\{\tilde{b}_l!\dm_l\}_{l=0}^s)$ is a biorthogonal wavelet filter bank with mixed dilation factors.
		
		\item[(ii)] For every $J\in\N$, \er{frdecomp} holds for all $v\in\dlrs{2}{1}{r}$, and moreover,\\* $(\DAS_J(\{b_l!\dm_l\}_{l=0}^s),\DAS_J(\{\tilde{b}_l!\dm_l\}_{l=0}^s))$ is a pair of biorthogonal systems in $\dlp{2}$ which satisfies the following biorthogonality relations:
		\be\label{das:ortho:1}\la \tilde{b}_{0,J;k'},b_{0,J;k}\ra=\td(k-k')I_r,\qquad \la \tilde{b}_{l,j;k},b_{0,J;k'}\ra=0 \qquad l=1,\dots,s,\quad j=1,\dots,J, \quad k,k'\in\dZ,\ee
		\be\label{das:ortho:2}	\la \tilde{b}_{l',j';k'},b_{l,j;k}\ra=\td(l-l')\td(j-j')\td(k-k'),\qquad l,l'=1,\dots,s,\quad j,j'=1,\dots,J, \quad k,k'\in\dZ.\ee
		\\

		***************************************************************\\
		\item[] If we further assume that both $\{b_l!\dm_l\}_{l=0}^s$ and $\{\tilde{b}_l!\dm_l\}_{l=0}^s$ have stability in $\dlp{2}$, then each of the above statements is equivalent to the following statement:
		
		\item[(iii)]\er{frdecomp} holds and there exist $C_1,C_2>0$ such that \er{das:frm} and \er{t:das:frm} hold for all $v\in\dlp{2}$ and $J\in\N$. Moreover,  $(\DAS_J(\{b_l!\dm_l\}_{l=0}^s),\DAS_J(\{\tilde{b}_l!\dm_l\}_{l=0}^s))$ is a pair of biorthogonal systems which satisfies \er{das:ortho:1} and \er{das:ortho:2}. 
		
	\end{enumerate}
\end{theorem}

\bp (i)$\Leftrightarrow$(ii): Suppose that (i) holds. Then, in particular, $(\{b_l!\dm_l\}_{l=0}^s,\{\tilde{b}_l!\dm_l\}_{l=0}^s)$ is a dual framelet filter bank with mixed dilation factors. Thus \er{frdecomp} holds for all $J\in\N$ and $v\in\dlrs{2}{1}{r}$ by Lemma~\ref{fraff}. Moreover, we have $\tilde{\cV}_J\cW_J=\ID_{(\dsq)^{1\times r}}$ for all $J\in\N$. Now the biorthogonality relations \er{das:ortho:1} and \er{das:ortho:2} follow straight away from \er{dft:key1} and the injectivity of $\tilde{\cV}_J$.\\

Conversely, suppose that (ii) holds, then \er{frdecomp} implies that  $(\{b_l!\dm_l\}_{l=0}^s,\{\tilde{b}_l!\dm_l\}_{l=0}^s)$ is a dual framelet filter bank with mixed dilation factors. Moreover the injectivity of $\cV_J$ follows from the biorthogonality relations \er{das:ortho:1} and \er{das:ortho:2} and the fact that
\be\label{vjexp}\tilde{\cV}_Jw=\sum_{k\in\dZ}w_{0,J}(k)\tilde{b}_{0,J;k}+\sum_{j=1}^J\sum_{l=1}^s\sum_{k\in\dZ}w_{l,j}(k)\tilde{b}_{l,j;k}\ee
holds for all $w=(w_{1,1},\dots,w_{1,J},\dots,w_{s,1},\dots,w_{s,J},w_{0,J})\in(\dsq)^{1\times(sJ+r)}$. This proves item (i).\\

Finally, by further assuming that both $\{b_l!\dm_l\}_{l=0}^s$ and $\{\tilde{b}_l!\dm_l\}_{l=0}^s$ have stability in $\dlp{2}$, it follows from Lemma~\ref{stadas} that each of items (i) and (ii) is equivalent to item (iii).

\ep

\section{Framelets and Wavelets in $\dLp{2}$ with Mixed dilation Factors}
\label{sec:fr:l2}

In this section, we discuss the connections between framelet filter banks and framelets in $\dLp{2}$. \\

\subsection{A Brief Review of Refinable Vector Functions}
First, we briefly review several definitions and results related to refinable functions obtained from finitely supported filters. The following result is well-known (see, e.g., \cite[Theorem 5.1.2]{hanbook}).
\begin{lemma}\label{ref0}Let $b_0\in\dlrs{0}{r}{r}$ be a finitely supported filter and let $\dm_0$ be a $d\times d$ dilation matrix. Suppose there exist $C_0,C_b,\tau>0$ and $u\in\dlrs{0}{r}{1}$ such that
	\be\label{mat:norm:a}\log_2\left(\vertiii{\wh{b_0}(0)}+C_0\right)\le\tau,\qquad  \|\wh{b_0}(\xi)\wh{u}(\xi)-\wh{u}(\dm_0^{\tp}\xi)\|\le C_b\|\xi\|^\tau,\qquad\xi\in[-\pi,\pi]^d,\ee
	where $\vertiii{\cdot}$ is some fixed matrix norm. Then
	\be\label{ref}\varphi(\xi):=\lim_{n\to\infty}\left(\prod_{j=1}^n\wh{a}((\dm_0^{-\tp})^j\xi)\right)\wh{u}((\dm_0^{-\tp})^n\xi),\qquad\xi\in\dR\ee
	is a well-defined bounded  measurable function which satisfies
	\be\label{limphi}\varphi(\xi)=\wh{u}(\xi)+\bo(\|\xi\|^\tau),\qquad \xi\to 0.\ee
	Moreover, there exists a compactly supported vector distribution $\psi^0$ such that $\wh{\psi^0}=\varphi$, and satisfies the following refinement equation:
	\be\label{ref:t}\psi^0(\cdot)=\sum_{k\in\dZ}b_0(k)\psi^0(\dm_0\cdot-k),\ee
	with the above series converging in the sense of tempered distributions.
\end{lemma}

\begin{definition}For any finitely supported filter $b_0\in\dlrs{0}{r}{r}$  satisfying \er{mat:norm:a} for some $C_0,C_b,\tau>0$ and $u\in\dlrs{0}{r}{1}$, the function $\phi$ which is defined via \er{ref} is called \textbf{a standard $\dm_0$-refinable function} associated to the filter $a$.
\end{definition}

The following proposition is an important result for our later study on framelets and wavelets in $\dLp{2}$.

\begin{prop}\label{refl2}Let $b_0\in\dlrs{0}{r}{r},b_1,\dots,b_s\in\dlrs{0}{1}{r}$ be finitely supported filters and let $\dm_0,\dots,\dm_s$ be $d\times d$ dilation matrices. Suppose there exist $C_0,C_b,\tau>0$ and $u\in\dlrs{0}{r}{1}$ such that \er{mat:norm:a} holds. Define a compactly supported standard refinable vector function $\psi^0$ via \er{ref} and define $\psi^1,\dots,\psi^s$ via
	\be\label{reff}\wh{\psi^l}(\xi)=\wh{b_l}(\dm_0^{-\tp}\xi)\wh{\psi^0}(\dm_0^{-\tp}\xi),\qquad\xi\in\dR, \quad l=1,\dots,s.
	\ee
	If $\{b_l!\dm_l\}_{l=0}^s$ has stability in $\dlp{2}$, then there exists $C>0$ such that 
	\be\label{phi:psi:bes}\sum_{k\in\dZ}\wh{\psi^0}(\xi+2\pi k)\ol{\wh{\psi^0}(\xi+2\pi k)}^{\tp}\le CI_r,\qquad \sum_{l=1}^s\sum_{k\in\dZ}\left|\wh{\psi^l}((\dm_l^{-\tp}\dm_0^{\tp})(\xi+2\pi k))\right|^2\le C\ee 
	for a.e. $\xi\in\dR$. 
\end{prop}

\bp   For $J\in\N$, let $\cV_J$ and $\cW_J$ be the $J$-level discrete framelet synthesis and analysis operators employing the filter bank $\{b_0!\dm_0,\dots,b_s!\dm_s\}$. By the stability of $\{b_l!\dm_l\}_{l=0}^s$ in $\dlp{2}$, there exists $K>0$ such that
\be\label{besselv}\|\cV_Jw\|_{\dlrs{2}{1}{r}}^2\le K\|w\|_{\dlrs{2}{1}{(sJ+r)}}^2,\qquad  w\in\dlrs{2}{1}{(sJ+r)},\quad J\in\N.\ee
Let $w=(0,\dots,0,v)\in \dlrs{2}{1}{(sJ+r)}$ with $v\in\dlrs{2}{1}{r}$. We have \be\label{sd:conv:b}\cV_Jw=\sd_{b_0,\dm_0}^Jv=\sd_{b_{0,J},\dm_0^J}v,\ee
where $b_{0,J}=(b_0\uparrow \dm_0^{J-1})*\dots*(b_0\uparrow \dm_0)*b_0$. Thus by Parseval's identity, we have
\be\label{sta:b}\begin{aligned}&\|\sd_{u,\dm_0}\cV_Jw\|_{\dlp{2}}^2\\
	=&(2\pi)^{-d}\int_{(\dm_0^{\tp})^J[0,2\pi)^d}\wh{v}(\xi)\wh{b_{0,J}}((\dm_0^{-\tp})^J\xi)\wh{u}((\dm_0^{-\tp})^J\xi)\ol{\wh{b_{0,J}}((\dm_0^{-\tp})^J\xi)\wh{u}((\dm_0^{-\tp})^J\xi)}^{\tp}\ol{\wh{v}(\xi)}^{\tp}d\xi.
\end{aligned}\ee
It now follows from Fatou's lemma that
\be\label{besselphi}\begin{aligned}&(2\pi)^{-d}\int_{\dR}\wh{v}(\xi)\wh{\psi^0}(\xi)\ol{\wh{\psi^0}(\xi)}^{\tp}\ol{\wh{v}(\xi)}^{\tp}d\xi\le\liminf_{J\to\infty}\|\sd_{u,\dm_0}\cV_Jw\|_{\dlp{2}}^2\\
	\le& \liminf_{J\to\infty}|\det(\dm_0)\|\|\cV_Jw\|_{\dlrs{2}{1}{(s+r)}}^2\|u\|_{\dlrs{1}{r}{1}}^2=K|\det(\dm_0)|\|u\|_{\dlrs{1}{r}{1}}^2\|v\|_{\dlrs{2}{1}{r}}^2
\end{aligned}\ee
for all $v\in\dlp{2}$. Note that 
\be\label{sta:b:2}\int_{\dR}\wh{v}(\xi)\wh{\psi^0}(\xi)\ol{\wh{\psi^0}(\xi)}^{\tp}\ol{\wh{v}(\xi)}^{\tp}d\xi=\int_{[0,2\pi)^{d}}\wh{v}(\xi)\left(\sum_{k\in\dZ}\wh{\psi^0}(\xi+2\pi k)\ol{\wh{\psi^0}(\xi+2\pi k)}^{\tp}\right)\ol{\wh{v}(\xi)}^{\tp}d\xi,\qquad v\in\dlrs{2}{1}{r}.\ee
It follows from \er{besselphi} and \er{sta:b:2} and the periodicity of $\sum_{k\in\dZ}\wh{\psi^0}(\xi+2\pi k)\ol{\wh{\psi^0}(\xi+2\pi k)}^{\tp}$ that 
\be\label{psi:0:br}\sum_{k\in\dZ}\wh{\psi^0}(\xi+2\pi k)\ol{\wh{\psi^0}(\xi+2\pi k)}^{\tp}\le K|\det(\dm_0)|\|u\|_{\dlrs{1}{r}{1}}^2I_r,\ee
for a.e. $\xi\in\dR$.\\

Similarly, choose $w=(0,\dots, w_{l,J},\dots,0)\in\dlrs{2}{1}{(sJ+r)}$ with $w_{l,J}\in\dlp{2}$ for some fixed $l\in\{1,\dots,s\}$, and using the fact that
\be\label{ref:psi:l}\wh{\psi^l}(\xi)=\lim_{J\to\infty}\wh{b_l}(\dm_0^{-\tp}\xi)\prod_{j=1}^{J-1}\wh{\psi^0}((\dm_0^{-\tp})^{j}\xi),\qquad l=1,\dots,s,\quad \xi\in\dR,\ee 
one can apply the above same arguments to conclude that there exists $K\tilde>0$ such that\\*  $\sum_{k\in\dZ}\left|\wh{\psi^l}((\dm_l^{-\tp}\dm_0^{\tp})(\xi+2\pi k))\right|^2\le \tilde{K}$ for all $l=1,\dots,s$ and a.e. $\xi\in\dR$. This completes the proof.

\ep

\subsection{Framlets and Wavelets in $\dLp{2}$ and Connections to Filter Banks}For a function (distribution) matrix $f:\dR\to\mathbb{C}^{s\times r}$ and a $d\times d$ invertible real matrix $U$, define
\be\label{f:dil}f_{U;k,n}(x):=|\det(U)|^{\frac{1}{2}}e^{-in\cdot Ux}f(Ux-k),\qquad  x,k,n\in\dR.\ee
In particular, define $f_{U;k}:=f_{U;k,0}$. It is straight forward to verify that
\be\label{f:ft:dil}\widehat{f_{U;k,n}}=\wh{f}_{U^{-\tp};n,k}.\ee

\begin{definition}Let $f:\dR\to\mathbb{C}^{s\times r}$ and $g:\dR\to\mathbb{C}^{t\times r}$ are matrices of measurable functions, and $U$ is a $d\times d$ invertible real matrix, define the \textbf{$U$-bracket product} of $f$ and $g$ via:
	\be\label{br:prod:u}[f,g]_{U}(x):=\sum_{k\in\dZ}f(x+2\pi U^{-1}k)\ol{g(x+2\pi U^{-1}k)}^{\tp},\ee
	whenever the series converges absolutely for a.e. $x\in\dR$. Denote $[f,g]:=[f,g]_{I_d}$. Note that $[f,g]_U$ is an $s\times t$ matrix of $2\pi U^{-1}\dZ$-periodic functions.
\end{definition}

We discuss some important properties of the bracket product.

\begin{lemma}\label{brfr}Let $f\in\dLrs{2}{s}{r},g\in\dLrs{2}{t}{r}$ and let $U$ be a $d\times d$ invertible real matrix. We have $[\wh{f},\wh{g}]_U\in (L_1(U^{-1}\dT))^{s\times t}$ and its Fourier series is
	\be\label{fs:dil}|\det(U)|\sum_{k\in\dZ}\la f,g(\cdot+U^{\tp}k)\ra_{\dLp{2}}e^{-ik\cdot Ux}.\ee
\end{lemma}

\bp For $f\in\dLrs{2}{s}{r},g\in\dLrs{2}{t}{r}$, we have 
\be\label{br:prod}\left|\left([\wh{f},\wh{g}]_U(\xi)\right)_{j,l}\right|\le\sum_{k=1}^r[\wh{f}_{s,k},\wh{f}_{s,k}]_U[\wh{g}_{k,l},\wh{g}_{k,l}]_U<\infty,\qquad\xi\in\dR.\ee
It follows from \er{br:prod} that
\be\label{br:prod:0}\int_{U^{-1}[0,2\pi)^d}\left|\left([\wh{f},\wh{g}]_U(\xi)\right)_{j,l}\right|d\xi\le\sum_{k=1}^r\|\wh{f}_{s,k}\|_{\dLp{2}}\|\wh{g}_{k,l}\|_{\dLp{2}}.\ee
Thus $[\wh{f},\wh{g}]_U\in (L_1(U^{-1}\dT))^{s\times t}$, and its $k$-th Fourier coefficient is given by
$$\begin{aligned}
&|\det(U)|(2\pi)^{-d}\int_{U^{-1}[0,2\pi)^d}\sum_{n\in\dZ}\wh{f}(\xi+2\pi U^{-1}n)\ol{\wh{g}(\xi+2\pi U^{-1}n)}^{\tp}e^{-ik\cdot U\xi}d\xi\\
=&|\det(U)|(2\pi)^{-d}\int_{\dR}\wh{f}(\xi)\ol{\wh{g}(\xi)e^{ik\cdot U\xi}}^{\tp}d \xi\\
=&|\det(U)|(2\pi)^{-d}\langle \wh{f},\wh{g}_{I_d;0,-U^{\tp}k}\ra_{\dLp{2}}\\
=&|\det(U)|\la f,g(\cdot+U^{\tp}k)\ra_{\dLp{2}}.
\end{aligned}$$

\ep

The following lemma is a generalization of \cite[Lemma 4.1.1]{hanbook} to high dimensions and arbitrary dilation factors. The result can be proved by following the lines of the proof of \cite[Lemma 4.1.1]{hanbook}.

\begin{lemma}\label{br}Let $U$ be an invertible $d\times d$ real matric, and $f,g:\dR\to\C,h,\tilde{h}:\dR\to\C^{r}$ be such that
	\be\label{f:br}\int_{\dR}\sum_{k\in\dZ}\|f(x)\ol{h(Ux)}^{\tp}\|\|f(x+2\pi U^{-1}k)\ol{h(Ux+2\pi k)}^{\tp}\|dx<\infty\ee
	and
	\be\label{g:br}\int_{\dR}\sum_{k\in\dZ}\|g(x)\ol{\tilde{h}(Ux)}^{\tp}\|\|g(x+2\pi U^{-1}k)\ol{\tilde{h}(Ux+2\pi k)}^{\tp}\|dx<\infty.\ee
	Then
	\be\label{fg:br}\sum_{k\in\dZ}\la f,h_{U;0,k}\ra_{\dLp{2}}\langle \tilde{h}_{U;0,k},g\ra_{\dLp{2}}=(2\pi)^d\int_{\dR}\sum_{k\in\dZ}f(x)\ol{h(Ux)}^{\tp}\tilde{h}(Ux+2\pi k)\ol{g(x+2\pi U^{-1}k)}dx.\ee
\end{lemma}
\bp
Note that
\be\label{f:h:br}\begin{aligned}&\int_{U^{-1}[-\pi,\pi)^d}\|[f,h(U\cdot)]_U(x)\|^2d x=\int_{\dR}f(x)\ol{h(Ux)}^{\tp}\ol{[f,h(U\cdot)]_U(x)}^{\tp}dx\\
	\le&\int_{\dR}\sum_{k\in\dZ}\|f(x)\ol{h(Ux)}^{\tp}\|\|f(x+2\pi U^{-1}k)\ol{h(Ux+2\pi k)}^{\tp}\|dx<\infty.
\end{aligned}\ee
Thus $[f,h(U\cdot)]_U\in (L_2(U^{-1}\dT))^{1\times r}$, and by Lemma~\ref{brfr}, its $k$-th Fourier coefficient is given by:
\be\label{f:h:br:fc}|\det(U)|^{\frac{1}{2}}(2\pi)^{-d}\la f,h_{U;0,-k}\ra_{\dLp{2}}.\ee
Similarly we have $[g,\tilde{h}(U\cdot)g]_U\in (L_2(U^{-1}\dT))^{1\times r}$, with its $k$-th Fourier coefficient equals  
\be\label{g:h:br:fc}|\det(U)|^{\frac{1}{2}}(2\pi)^{-d}\la g,\tilde{h}_{U;0,-k}\ra_{\dLp{2}}.\ee
By Parseval's identity:
\be\label{f:g:frexp}\begin{aligned}
	\sum_{k\in\dZ}\la f,h_{U;0,k}\ra_{\dLp{2}}\la \tilde{h}_{U;0,k},g\ra_{\dLp{2}}=&(2\pi)^d\int_{U^{-1}[-\pi,\pi)^d}[f,h(U\cdot)]_U(x)[\tilde{h}(U\cdot),g]_U(x)dx\\
	=&(2\pi)^d\int_{\dR}\sum_{k\in\dZ}f(x)\ol{h(Ux)}^{\tp}\tilde{h}(Ux+2\pi k)\ol{g(x+2\pi U^{-1}k)}dx.
\end{aligned}\ee

\ep

We now introduce the notion of an affine system in $\dLp{2}$. 

\begin{definition}Let $\psi^0=[\psi^0_1,\dots,\psi^0_r]^{\tp}\in\dLr{2}{r},\psi^1,\dots,\psi^s\in \dLp{2}$ and $\dm_0,\dots,\dm_s$ be $d\times d$ dilation matrices. Define an affine system via
	\be\label{afs}\begin{aligned}\AS(\{\psi^l!\dm_l\}_{l=0}^s):=&\{\psi^0_q(\cdot-k):q=1,\dots,r; k\in\dZ\}\\
		&\cup\{|\det(\dm_0^{-1}\dm_l)|^{\frac{1}{2}}\psi^l_{\dm_0^j;\dm_0^{-1}\dm_lk}:l=1,\dots,s;j\in\N_0;k\in\dZ\}.
	\end{aligned}\ee
\end{definition}

Now we introduce the notion of framelets and wavelets in $\dLp{2}$ with mixed dilation factors.\\

%
%

Framelets/wavelets in $\dLp{2}$ are naturally connected with discrete framelets framelet/wavelet filter banks in $\dlp{2}$. The following result connects discrete affine systems in $\dlp{2}$ with affine systems in $\dLp{2}$.\\

\begin{prop}\label{stafr}Let $b_0\in\dlrs{0}{r}{r},b_1,\dots,b_s\in\dlrs{0}{1}{r}$ be finitely supported filters and let $\dm_0,\dots,\dm_s$ be $d\times d$ dilation matrices. Suppose that $\psi^0\in\dLr{2}{r}$ is a refinable vector function associated to $b_0$ satisfying \er{ref:t}. Define $\psi^1,\dots,\psi^s$ via \er{reff} for $l=1,\dots,s$.  Suppose in addition that $\psi^0\in\dLr{2}{r}$. For $f\in\dLp{2}$, define
	\be\label{dasas0}w_{l,j}(k)=\la f, |\det(\dm_0^{-1}\dm_l)|^{\frac{1}{2}}\psi^l_{\dm_0^j;\dm_0^{-1}\dm_lk}\ra_{\dLp{2}},\qquad k\in\dZ, \quad l=0,\dots,s, \quad j\in\N_0.\ee
	The following statements hold:
	\begin{enumerate}
		
		\item[(i)] For every $j\in\N_0, l=0,\dots,s$ and $ k\in\dZ$, we have \be\label{dasas}w_{l,j}(k)=\tz_{b_l\dm_l}w_{0,j+1}(k).\ee
		
		\item[(ii)] For every $J,j\in\N_0$ with $J>j$, $l=0,\dots,s$ and $k\in\dZ$, we have
		\be\label{da:coef}w_{l,J}(k)=\la w_{0,J},b_{l,J-j;k}\ra_{\dlp{2}},\ee
		where $b_{l,j;k}$ is defined via \er{b:l:j:k}.
		
		\item[(iii)] If further assume that $\{b_l!\dm_l\}_{l=0}^s$ has stability in $\dlp{2}$, then $\AS(\{\psi^l!\dm_l\}_{l=0}^s)$ is a Bessel sequence in  $\dLp{2}$.

	\end{enumerate}
\end{prop}

\bp Note that \er{reff} is equivalent to
\be\label{reff:td}\psi^l(x)=|\det(\dm_l)|\sum_{k\in\dZ}b_l(k)\psi^0(\dm_0x-k),\qquad x\in\dR,\quad l=0,\dots,s.\ee
Thus we see that $\psi^1,\dots,\psi^s\in\dLp{2}$. By calculation:
$$\begin{aligned} w_{l,j}(k)&=\int_{\dR}f(x)|\det(\dm_0^{-1}\dm_l)|^{\frac{1}{2}}\ol{\psi^l_{\dm_0^j;\dm_0^{-1}\dm_lk}(x)}^{\tp}dx\\
&=|\det(\dm_0^{j-1}\dm_l)|^{\frac{1}{2}}\int_{\dR}f(x)\ol{\psi^l(\dm_0^jx-\dm_0^{-1}\dm_lk)}^{\tp}dx\\
&=\sum_{m\in\dZ}|\det(\dm_0^{j+1}\dm_l)|^{\frac{1}{2}}\left(\int_{\dR}f(x)\ol{\psi^0(\dm_0^{j+1}x-\dm_lk-m)}^{\tp}dx\right)\ol{\wh{b_l}(m)}^{\tp}\\
&=\sum_{m\in\dZ}|\det(\dm_l)|^{\frac{1}{2}}\la f,\psi^0_{\dm_0^{j+1};\dm_lk+m}\ra_{\dLp{2}}\ol{\wh{b_l}(m)}^{\tp}\\
&=\sum_{m\in\dZ}|\det(\dm_l)|^{\frac{1}{2}}w_{0,j+1}(\dm_lk+m)\ol{\wh{b_l}(m)}^{\tp}\\
&=\tz_{b_l,\dm_l}w_{0,j+1}(k),
\end{aligned}$$
for all $k\in\dZ$, $j\in\N_0$ and $l=0,\dots,s$. This proves item (i)\\

Next, define $b_{l,j}$ via \er{blj} for all $j\in\N$ and $l=0,\dots,s$. Use item (i), we have
$$\begin{aligned}w_{l,j}(k)&=\tz_{b_l,\dm_l}\tz_{b_0,\dm_0}^{J-j-1}w_{0,J}(k)=\tz_{(b_l\uparrow \dm_0^{J-j-1})*(b_0\uparrow \dm_0^{J-j-2})*\dots*(b_0\uparrow \dm_0)*b_0,\dm_0^{J-j-1}\dm_l}w_{0,J}(k)=\tz_{b_{l,J-j},\dm_0^{J-j-1}\dm_l}w_{0,J}(k)\\
&=|\det(\dm_l)|^{\frac{1}{2}}|\det(\dm_0)|^{\frac{J-j-1}{2}}\sum_{m\in\dZ}w_{0,J}(m)\ol{b_{l,J-j}(m-\dm_0^{J-j-1}\dm_lk)}^{\tp}=\la w_{0,J}, b_{l,J-j;k}\ra_{\dlp{2}},
\end{aligned}$$
whenever $J,j\in\N_0$ and $J>j$. This proves item (ii).\\

Finally, we prove item (iii). By items (i) and (ii), we have
\be\label{cdafs}\begin{aligned}&\sum_{k\in\dZ}\|\la f,\psi^0_{I_d;k}\ra_{\dLp{2}}\|^2+\sum_{j=0}^{J-1}\sum_{k\in\dZ}|\la f,|\det(\dm_0^{-1}\dm_l)|^{\frac{1}{2}}\psi^l_{\dm_0^j;\dm_0^{-1}\dm_lk}\ra_{\dLp{2}}|^2\\
	=&\sum_{k\in\dZ}\|w_{0,0}(k)\|^2+\sum_{j=0}^{J-1}\sum_{k\in\dZ}|w_{l,j}|^2=\sum_{k\in\dZ}\|\la w_{0,J},b_{0,J;k}\ra_{\dlp{2})}\|^2+\sum_{j=0}^{J-1}\sum_{k\in\dZ}|\la w_{0,J},b_{l,J-j;k}\ra_{\dlp{2}}|^2,
\end{aligned}\ee
for all $J\in\N$. Since $\{b_l!\dm_l\}_{l=0}^s$ has stability in $\dlp{2}$, Lemma~\ref{stadas} yields that there exists $C>0$ such that
\be\label{da:bes}\sum_{k\in\dZ}\|\la w_{0,J},b_{0,J;k}\ra_{\dlp{2})}\|^2+\sum_{j=0}^{J-1}\sum_{k\in\dZ}|\la w_{0,J},b_{l,J-j;k}\ra_{\dlp{2}}|^2\le C\|w_{0,J}\|_{\dlp{2}}^2,\qquad J\in\N.\ee
Moreover, by Proposition~\ref{refl2}, there exists $K>0$ such that $[\wh{\psi^0}^{\tp},\wh{\psi^0}^{\tp}](\xi)\le K$ for a.e. $\xi\in\dR$, i.e., $[\wh{\psi^0}^{\tp},\wh{\psi^0}^{\tp}]\in \dTLp{\infty}$. Thus $[\wh{f},\wh{\psi^0}]\in \dTLp{2}$ for all $f\in \dLp{2}$. Now use Lemma~\ref{brfr} and Parseval's identity, we have
\be\label{w:ub}\begin{aligned}
	\|w_{0,J}\|_{\dlp{2}}^2&=\sum_{k\in\dZ}|\la f,\psi^0_{\dm_0^J;k}\ra_{\dLp{2}}|^2=\sum_{k\in\dZ}|\la f_{\dm_0^{-J};0},\psi^0(\cdot-k)\ra_{\dLp{2}}|^2\\
	&=(2\pi)^{-d}\int_{[0,2\pi)^d}|[\wh{f_{(\dm_0^{\tp})^J;0}},\wh{\psi^0}](\xi)|^2d\xi \\
	&\le (2\pi)^{-d}\int_{[0,2\pi)^d}[\widehat{f_{(\dm_0^{T})^J;0}},\wh{f_{(\dm_0^{T})^J;0}}](\xi)[\wh{\psi^0}^{\tp},\wh{\psi^0}^{\tp}](\xi)d\xi\\
	&\le (2\pi)^{-d}K\int_{[0,2\pi)^d}[\wh{f_{(\dm_0^{\tp})^J;0}},\wh{f_{(\dm_0^{\tp})^J;0}}](\xi)d\xi\\
	&=(2\pi)^{-d}K\|\wh{f_{(\dm_0^{\tp})^J;0}}\|_{\dLp{2}}^2=K\|f\|_{\dLp{2}}^2.
\end{aligned}\ee
It follows from \er{cdafs}, \er{da:bes} and \er{w:ub} that 
\be\label{f:ft:ub}\sum_{k\in\dZ}\|\la f,\psi^0_{I_d;k}\ra_{\dLp{2}}\|^2+\sum_{j=0}^{J-1}\sum_{k\in\dZ}|\la f,|\det(\dm_0^{-1}\dm_l)|^{\frac{1}{2}}\psi^l_{\dm_0^j;\dm_0^{-1}\dm_lk}\ra_{L^2(\mathbb{R}^d)}|^2\le CK\|f\|_{\dLp{2}}^2\ee
for all $f\in \dLp{2}$ and $J\in\N$. By letting $J\to\infty$, we see that $\AS(\{\phi^l!\dm_l\}_{l=0}^s)$ is a Bessel sequence in $\dLp{2}$. This proves item (iii).

\ep

We are now ready to state the main theorem, which connects discrete framelet filter banks and framelets in $\dlp{2}$.

\begin{theorem}\label{charfrl2}Let $b_0,\tilde{b}_0\in\dlrs{0}{r}{r}, b_1,\dots,b_s,\tilde{b}_1,\dots\tilde{b}_s\in\dlrs{0}{1}{r}$ be finitely supported filters and let $\dm_0,\dots,\dm_s$ be $d\times d$ dilation matrices. Suppose $\psi^0,\tilde{\psi}^0\in\dLr{2}{r}$ are compactly supported standard refinable vector functions satisfying
	\be\label{ref:f}\wh{\psi^0}(\dm_0^{\tp}\xi)=\wh{b_0}(\xi)\wh{\psi^0}(\xi),\qquad\wh{\tilde{\psi}^0}(\dm_0^{\tp}\xi)=\wh{\tilde{b}_0}(\xi)\wh{\tilde{\psi}^0}(\xi),\qquad\xi\in\dR,\ee
	and $\ol{\wh{\psi^0}(0)}^{\tp}\wh{\tilde{\psi}(0)}=1$. Define $\psi^1,\dots,\psi^s,\tilde{\psi}^1,\dots,\tilde{\psi}^s\in\dLp{2}$ via
	\be\label{aff}\wh{\psi^l}(\xi)=\wh{b_l}(\dm_0^{-\tp}\xi)\wh{\psi^0}(\dm_0^{-\tp}\xi),\qquad  \wh{\tilde{\psi}^l}(\xi)=\wh{\tilde{b}_l}(M_0^{-\tp}\xi)\wh{\tilde{\psi}^0}(\dm_0^{-\tp}\xi),\qquad\xi\in\dR.\ee

	Then
	
	\begin{enumerate}
		
		\item[(i)] $(\{b_l!\dm_l\}_{l=0}^s,\{\tilde{b}_l!\dm_l\}_{l=0}^s)$ is a dual framelet filter bank with mixed dilation factors. Furthermore, both $\{b_l!\dm_l\}_{l=0}^s$ and $\{\tilde{b}_l!\dm_l\}_{l=0}^s$ have stability in $\dlp{2}$,\\*

		\item[] implies	\\*

		\item[(ii)] $(\{\psi^l!\dm_l\}_{l=0}^s,\{\tilde{\psi}^l!\dm_l\}_{l=0}^s)$ is a dual framelet in $\dLp{2}$. Furthermore, there exists $C>0$ such that
		\be\label{ulbs}[\wh{\psi^0},\wh{\psi^0}](\xi)+[\wh{\tilde{\psi}^0},\wh{\tilde{\psi}^0}](\xi)\le CI_r,\qquad\text{a.e. }\xi\in\dR.\ee
	Here, for two square matrices $A,B$ of the same dimension, the relation $A\le B$ means $B-A$ is positive semi-definite.	
	\end{enumerate}
	
	Conversely, if in addition
	\be\label{llbs}[\wh{\psi^0},\wh{\psi^0}](\xi)\ge C'I_r,\qquad[\wh{\tilde{\psi}^0},\wh{\tilde{\psi}^0}](\xi)\ge C'I_r,\qquad\text{a.e. }\xi\in\dR,\ee
	for some constant $C'>0$, then item (ii) implies item (i).
\end{theorem}

\bp Define 
\be\label{L2:D}D:=\{f\in \dLp{2}:\wh{f}\in \dDC\}.\ee
By Lemma~\ref{br}, for every $f,g\in D$ and $j\in\N_0$, we have
\be\label{psi0j}\begin{aligned}
	&\sum_{k\in\dZ}\la \wh{f},\wh{\psi^0}_{((\dm_0)^{-\tp})^j;0,k}\ra_{\dLp{2}}\la \wh{\tilde{\psi}^0}_{((\dm_0)^{-\tp})^j;0,k,\wh{g}}\ra_{\dLp{2}}\\
	=&(2\pi)^d\int_{\dZ}\sum_{k\in\dZ}\wh{f}(\xi)\ol{\wh{\psi^0}((\dm_0^{-\tp})^j\xi)}^{\tp}\wh{\tilde{\psi}^0}((\dm_0^{-\tp})^j\xi+2\pi k)\ol{\wh{g}(\xi+2\pi(\dm_0^{\tp})^jk)}d\xi.
\end{aligned}\ee
Choose $j$ sufficiently large such that $\wh{f}(\xi)\hat{g}(\xi+2\pi(\dm_0^\tp)^j\xi)=0$ for all $\xi\in\dR$ and $k\in\dZ\setminus\{0\}$, we have
\be\label{psi0j:0}\sum_{k\in\dZ}\la \wh{f},\wh{\psi^0}_{((\dm_0)^{-\tp})^j;0,k}\ra_{\dLp{2}}\la \wh{\tilde{\psi}^0}_{((\dm_0)^{-\tp})^j;0,k,\wh{g}}\ra_{\dLp{2}}
=(2\pi)^d\int_{\dR}\wh{f}(\xi)\ol{\wh{\psi^0}((\dm_0^{-\tp})^j\xi)}^{\tp}\wh{\tilde{\psi}^0}((\dm_0^{-\tp})^j\xi)\ol{\wh{g}(\xi)}d\xi.\ee
By Lemma~\ref{ref0}, $\wh{\psi^0}$ and $\wh{\tilde{\psi}^0}$ are vectors of bounded measurable functions and
\be\label{psi:0:lim}\lim_{j\to\infty}\ol{\wh{\psi^0}((\dm_0^{-\tp})^j\xi)}^{\tp}\wh{\tilde{\psi}^0}((\dm_0^{-\tp})^j\xi)=\ol{\wh{\psi^0}(0)}^{\tp}\wh{\tilde{\psi}^0}(0)=1,\qquad\xi\in\dR.\ee 
It follows from the Dominated convergence theorem that
\be\label{psi:0:dc}\lim_{j\to\infty}\sum_{k\in\dZ}\la \hat{f},\wh{\psi^0}_{((\dm_0)^{-\tp})^j;0,k}\ra_{\dLp{2}}\la \wh{\tilde{\psi}^0}_{((\dm_0)^{-\tp})^j;0,k,\wh{g}}\ra_{\dLp{2}}=(2\pi)^d\int_{\dR}\wh{f}(\xi)\ol{\wh{g}(\xi)}d\xi=(2\pi)^d\la\wh{f},\wh{g}\ra_{\dLp{2}}.\ee
Define 
\be\label{eta:psi}\eta^l:=\psi^l(\dm_0^{-1}\dm_l\cdot),\qquad\tilde{\eta}^l=\tilde{\psi}^l(\dm_0^{-1}\dm_l\cdot),\qquad l=1,\dots,s.\ee
We have
\be\label{psi:eta}\psi^l(\cdot-\dm_0^{-1}\dm_lk)=\eta^l(\dm_l^{-1}\dm_0\cdot-k),\qquad \wh{\eta^l}=|\det(\dm_0^{-1}\dm_l)|^{-1}\wh{\psi^l}(\dm_0^\tp\dm_l^{-\tp}\cdot),\ee
\be\label{tpsi:eta}\tilde{\psi}^l(\cdot-\dm_0^{-1}\dm_lk)=\tilde{\eta}^l(\dm_l^{-1}\dm_0\cdot-k),\qquad \wh{\tilde{\eta}^l}=|\det(\dm_0^{-1}\dm_l)|^{-1}\wh{\tilde{\psi}^l}(\dm_0^{\tp}\dm_l^{-\tp}\cdot),\ee
for all $l=1,\dots,s$. For $f,g\in D$, Lemma~\ref{br} yields
\be\label{f:g:eta}\begin{aligned}
	&|\det(\dm_0^{-1}\dm_l)|^2\sum_{k\in\dZ}\la \wh{f},\wh{\eta^l}_{\dm_l^\tp M_0^{-\tp};0,k}\ra_{\dLp{2}}\la \wh{\tilde{\eta}^l}_{\dm_l^\tp M_0^{-\tp};0,k},\wh{g}\ra_{\dLp{2}}\\
	=&(2\pi)^d|\det(\dm_0^{-1}\dm_l)|^2\int_{\dR}\sum_{k\in\dZ}\wh{f}(\xi)\ol{\wh{\eta^l}(\dm_l^\tp\dm_0^{-\tp}\xi)}^{\tp}\wh{\tilde{\eta}^l}(\dm_l^\tp\dm_0^{-\tp}\xi+2\pi k)\ol{\wh{g}(\xi+2\pi(\dm_0^\tp \dm_l^{-\tp}) k)}d\xi.\\
	=&(2\pi)^d\int_{\dR}\sum_{k\in\dZ}\wh{f}(\xi)\ol{\wh{\psi^l}(\xi)}^{\tp}\wh{\tilde{\psi}^l}(\xi+2\pi (\dm_0^\tp\dm_l^{-tp}) k)\ol{\wh{g}(\xi+2\pi(\dm_0^\tp\dm_l^{-tp}) k)}d\xi.\\
	=&(2\pi)^d\int_{\dR}\sum_{k\in\dZ}\wh{f}(\xi) \ol{\wh{b_l}(\dm_0^{-\tp}\xi)\wh{\psi^0}(\dm_0^{-\tp}\xi)}^{\tp}\wh{\tilde{b}_l}(\dm_0^{-\tp}\xi+2\pi (\dm_l^{-\tp}) k)\wh{\tilde{\psi}^0}(\dm_0^{-\tp}\xi+2\pi (\dm_l^{-\tp}) k)\\
	&\quad\times\ol{\wh{g}(\xi+2\pi(\dm_0^\tp\dm_l^{-\tp})k)}d\xi.\\
	=&(2\pi)^d\int_{\dR}\wh{f}(\xi)\ol{\wh{\psi^0}(\dm_0^{-\tp}\xi)}^{\tp}\left(\sum_{\omega\in\Omega_l}\ol{\wh{b_l}(\dm_0^{-\tp}\xi)}\wh{\tilde{b}_l}(\dm_0^{-\tp}\xi+2\pi\omega)\right)\\
	&\quad\times\sum_{k\in\dZ}\wh{\tilde{\psi}^0}(\dm_0^{-\tp}\xi+2\pi\omega+2\pi k)\ol{\wh{g}(\xi+2\pi \dm_0\omega+2\pi \dm_0k)}d\xi.
\end{aligned}\ee
Similarly:
\be\begin{aligned}
	&\sum_{k\in\dZ}\la \wh{f},\wh{\psi^0}_{I_d;0,k}\ra_{\dLp{2}}\la \wh{\tilde{\psi}^0}_{I_d;0,k},\wh{g}\ra_{\dLp{2}}\\
	=&(2\pi)^d\int_{\dR}\wh{f}(\xi)\ol{\wh{\psi^0}(\dm_0^{-\dm}\xi)}^{\tp}\left(\sum_{\omega\in\Omega_0}\ol{\wh{b_0}(\dm_0^{-\dm}\xi)}\wh{\tilde{b}_0}(\dm_0^{-\tp}\xi+2\pi\omega)\right)\\
	&\quad\times\sum_{k\in\dZ}\wh{\tilde{\psi}^0}(\dm_0^{-\tp}\xi+2\pi\omega+2\pi k)\ol{\wh{g}(\xi+2\pi \dm_0\omega+2\pi \dm_0k)}d\xi.
\end{aligned}\ee
It follows that
\be\label{prexp}\begin{aligned}&\sum_{k\in\dZ}\la \wh{f},\wh{\psi^0}_{I_d;0,k}\ra_{\dLp{2}}\la \wh{\tilde{\psi}^0}_{I_d;0,k},\wh{g}\ra_{\dLp{2}}+\sum_{l=1}^s|\det(\dm_0^{-1}\dm_l)|^2\sum_{k\in\dZ}\la \wh{f},\wh{\eta^l}_{\dm_l^\tp\dm_0^{-\tp};0,k}\ra_{\dLp{2}}\la \wh{\tilde{\eta}^l}_{\dm_l^\tp\dm_0^{-\tp};0,k},\wh{g}\ra_{\dLp{2}}\\
	=&(2\pi)^d\int_{\dR}\wh{f}(\xi)\ol{\wh{\psi^0}(\dm_0^{-\tp}\xi)}^{\tp}\left(\sum_{\omega\in\Omega}\sum_{l=0}^s\chi_{\Omega_l}(\omega)\ol{\wh{b_l}(\dm_0^{-\tp}\xi)}\wh{\tilde{b}_l}(\dm_0^{\tp}\xi+2\pi\omega)\right)\\
	&\quad\times\sum_{k\in\dZ}\wh{\tilde{\psi}^0}(\dm_0^{-\tp}\xi+2\pi\omega+2\pi k)\ol{\wh{g}(\xi+2\pi \dm_0\omega+2\pi \dm_0k)}d\xi
\end{aligned}
\ee
for all $f,g\in D$.\\

Suppose item (i) holds. By \er{pr:frt:mix},\er{psi0j} and \er{prexp}, we have 
\be\label{pr:fr}\begin{aligned}&\sum_{k\in\dZ}\la \wh{f},\wh{\psi^0}_{I_d;0,k}\ra_{\dLp{2}}\la \wh{\tilde{\psi}^0}_{I_d;0,k},\wh{g}\ra_{\dLp{2}}+\sum_{l=1}^s|\det(\dm_0^{-1}\dm_l)|^2\sum_{k\in\dZ}\la \wh{f},\wh{\eta^l}_{\dm_l^\tp\dm_0^{-\tp};0,k}\ra_{\dLp{2}}\la \wh{\tilde{\eta}^l}_{\dm_l^\tp\dm_0^{-\tp};0,k},\wh{g}\ra_{\dLp{2}}\\
	=&(2\pi)^d\int_{\dR}\wh{f}(\xi)\ol{\wh{\psi^0}(\dm_0^{-\tp}\xi)}^{\tp}\sum_{k\in\dZ}\wh{\tilde{\psi}^0}(\dm_0^{-\tp}\xi+2\pi k)\ol{\wh{g}(\xi+2\pi \dm_0k)}d\xi\\
	=&\sum_{k\in\dZ}\la \wh{f},\wh{\psi^0}_{\dm_0^{-\tp};0,k}\ra_{\dLp{2}}\la \wh{\tilde{\psi}^0}_{\dm_0^{-\tp};0,k},\wh{g}\ra_{\dLp{2}},
\end{aligned}\ee
for all $f,g\in D$. Note that 
\be\label{scale}\la \wh{f}_{U^{-1};0},\wh{g}\ra_{\dLp{2}}=\la \wh{f},\wh{g}_{U;0}\ra_{\dLp{2}},\qquad f,g\in\dLp{2}.\ee
It follows from \er{pr:fr} that
\be\label{casc1}\begin{aligned}&\sum_{k\in\dZ}\la \wh{f},\wh{\psi^0}_{(\dm_0^{-\tp})^j;0,k}\ra_{\dLp{2}}\la \wh{\tilde{\psi}^0}_{(\dm_0^{-\tp})^j;0,k},\wh{g}\ra_{\dLp{2}}\\
	&\quad+\sum_{l=1}^s|\det(\dm_0^{-1}\dm_l)|^2\sum_{k\in\dZ}\la \wh{f},\wh{\eta^l}_{\dm_l^\tp(\dm_0^{-\tp})^j;0,k}\ra_{\dLp{2}}\la \wh{\tilde{\eta}^l}_{\dm_l^\tp(\dm_0^{-\tp})^j;0,k},\wh{g}\ra_{\dLp{2}}\\
	=&\sum_{k\in\dZ}\la \wh{f},\wh{\psi^0}_{(\dm_0^{-\tp})^{j+1};0,k}\ra_{\dLp{2}}\la \wh{\tilde{\psi}^0}_{(\dm_0^{-\tp})^{j+1};0,k},\wh{g}\ra_{\dLp{2}},
\end{aligned}\ee
for all $f,g\in D$ and $j\in\mathbb{N}_0$. So for $m,n\in\mathbb{N}_0$ with $m>n$, we have
\be\label{casc2}\begin{aligned}&\sum_{k\in\dZ}\la \wh{f},\wh{\psi^0}_{(\dm_0^{-\tp})^n;0,k}\ra_{\dLp{2}}\la \wh{\tilde{\psi}^0}_{(\dm_0^{-\tp})^n;0,k},\wh{g}\ra_{\dLp{2}}\\
	&\quad+\sum_{j=n}^{m}\sum_{l=1}^s|\det(\dm_0^{-1}\dm_l)|^2\sum_{k\in\dZ}\la \wh{f},\wh{\eta^l}_{\dm_l^\tp(\dm_0^{-\tp})^j;0,k}\ra_{\dLp{2}}\la \wh{\tilde{\eta}^l}_{\dm_l^\tp(\dm_0^{-\tp})^j;0,k},\wh{g}\ra_{\dLp{2}}\\
	=&\sum_{k\in\dZ}\la \wh{f},\wh{\psi^0}_{(\dm_0^{-\tp})^{m+1};0,k}\ra_{\dLp{2}}\la \wh{\tilde{\psi}^0}_{(\dm_0^{-\tp})^{m+1};0,k},\wh{g}\ra_{\dLp{2}}.
\end{aligned}\ee
By letting $m\to\infty$, \er{psi:0:dc} and \er{casc2} yield
\be
\label{casc3}\begin{aligned}&\sum_{k\in\dZ}\la \wh{f},\wh{\psi^0}_{(\dm_0^{-\tp})^n;0,k}\ra_{\dLp{2}}\la \wh{\tilde{\psi}^0}_{(\dm_0^{-\tp})^n;0,k},\wh{g}\ra_{\dLp{2}}\\
	&\quad+\sum_{j=n}^{\infty}\sum_{l=1}^s|\det(\dm_0^{-1}\dm_l)|^2\sum_{k\in\dZ}\la \wh{f},\wh{\eta^l}_{\dm_l^\tp(\dm_0^{-\tp})^j;0,k}\ra_{\dLp{2}}\la \wh{\tilde{\eta}^l}_{\dm_l^\tp(\dm_0^{-\tp})^j;0,k},\wh{g}\ra_{\dLp{2}}\\
	=&(2\pi)^d\la\wh{f},\wh{g}\ra_{\dLp{2}},\end{aligned}\ee
for all $f,g\in D$ and $n\in\N_0$. By Plancherel's theorem and letting $n=0$ in \eqref{casc3}, we see that \er{dual:frt} holds for all $f,g\in D$. Since $D$ is dense in $\dLp{2}$, we conclude that \er{dual:frt} holds for all $f,g\in \dLp{2}$. Next, by item (iii) of Proposition~\ref{stafr}, the stability of $\{b_l!\dm_l\}_{l=0}^s$ and $\{\tilde{b}_l!\dm_l\}_{l=0}^s$ implies that both $\AS(\{\psi^l!\dm_l\}_{l=0}^s)$ and $\AS(\{\tilde{\psi}^l!\dm_l\}_{l=0}^s)$ are Bessel sequences. Thus by theory of Hilbert spaces (see e.g. \cite[Theorem 4.2.5]{hanbook}), $(\AS(\{\tilde{\psi}^l!\dm_l\}_{l=0}^s),\AS(\{\psi^l!\dm_l\}_{l=0}^s))$ is a pair of dual frames for $\dLp{2}$. Finally, \er{ulbs} follows right away from the item (iii) of Proposition~\ref{refl2}. This proves the direction (i) $\Rightarrow$ (ii).\\

Conversely, suppose that $(\{\psi^l!\dm_l\}_{l=0}^s,\{\tilde{\psi}^l!\dm_l\}_{l=0}^s)$ is a dual framelet in $\dLp{2}$. It follows from \er{psi0j}, \er{prexp} and \er{pr:fr} that
\be
\label{prexp1}\begin{aligned}0=&\int_{\dR}\wh{f}(\xi)\ol{\wh{\psi^0}(\dm_0^{-\tp}\xi)}^{\tp}\left(\sum_{\omega\in\Omega}\sum_{l=0}^s\chi_{\Omega_l}(\omega)\ol{\wh{b_l}(\dm_0^{-\tp}\xi)}\wh{\tilde{b}_l}(\dm_0^{-\tp}\xi+2\pi\omega)-\td(\omega)I_r\right)\\
	&\quad\times\sum_{k\in\dZ}\wh{\tilde{\psi}^0}(\dm_0^{-\tp}\xi+2\pi\omega+2\pi k)\ol{\wh{g}(\xi+2\pi \dm_0\omega+2\pi \dm_0k)}d\xi,\end{aligned}\ee
for all $f,g \in\dLp{2}$By using the similar argument as in the proof of (ii) $\Leftarrow$ (iii) in Theorem~\ref{pr:mix} (or see the proof of \cite[Lemma 5]{han10:0}), one can conclude that 
\be\label{gpr}\ol{\wh{\psi^0}(\xi)}^{\tp}\left(\sum_{l=0}^s\chi_{\Omega_l}(\omega)\ol{\wh{b_l}(\xi)}\wh{\tilde{b}_l}(\xi+2\pi\omega)-\td(\omega)I_r\right)\wh{\tilde{\psi}^0}(\xi+2\pi\omega+2\pi k)=0\end{equation}
for a.e. $\xi\in\dR$. Note that the condition \er{llbs} implies that
\be\label{span:psi}\text{span}\left\{\wh{\psi^0}(\xi+2\pi k):k\in\dZ\right\}=\C^r,\qquad\left\{\wh{\psi^0}(\xi+2\pi k):k\in\dZ\right\}=\C^r,\qquad \text{a.e. }\xi\in\dR.\ee
Thus \er{gpr} yields \er{pr:frt:mix}, and we conclude that $(\{b_l!\dm_l\}_{l=0}^s,\{\tilde{b}_l!\dm_l\}_{l=0}^s)$ is a dual framelet filter bank with mixed dilation factors. It remains to prove that both $\{b_l!\dm_l\}_{l=0}^s$ and $\{\tilde{b}_l!\dm_l\}_{l=0}^s$ have stability in $\dlp{2}$ under the additional assumptions that \er{ulbs} and \er{llbs} hold. Define the \emph{shift invariant space} generated by $\psi^0=[\psi^0_1,\dots,\psi^0_r]^{\tp}$ via
\be\label{shinv}S(\psi^0):=\ol{\text{span}\{\psi^0_l(\cdot-k):l=1,\dots,r;  k\in\dZ\}}.\ee
By \er{ulbs}, \er{llbs} and \cite[Theorem 4.4.12]{hanbook}, $\{\psi^0_l(\cdot-k):l=1,\dots,r;  k\in\dZ\}$ is a Riesz basis of $S(\psi^0)$. This means that the map
\be\label{ana:shinv}\cW_{\psi^0}:S(\psi^0)\to\dlr{2}{r},\qquad \cW_{\psi^0}(f)=\{\la f,\psi^0(\cdot-k)\ra_{\dLp{2}}\}_{k\in\dZ}\ee
is a well-defined bounded isomorphism. Thus for each $v\in\dlr{2}{r}$, there exists a unique $h^v\in S(\psi^0)$ such that $\cW_{\psi^0}h^v=v$. Use \cite[Proposition 4.4.13]{hanbook}, we have
\be\label{ub:v}\|v\|_{\dlr{2}{r}}^2=\sum_{k\in\dZ}|\la h_v,\psi^0(\cdot-k)\ra_{\dLp{2}}|^2\ge C_2\|h^v\|_{\dLp{2}}^2,\ee
for some constant $C_2>0$. On the other hand, let $\cW_J$ be the $J$-level discrete framelet analysis operator employing the filter bank $\{b_l!\dm_l\}_{l=0}^s$. By \er{sta:fb}, \er{cdafs} and the fact that $\la h^v_{\dm_0^J;0},\psi^0_{\dm_0^J;k}\ra_{\dLp{2}}=\la h^v,\psi^0(\cdot-k)\ra_{\dLp{2}}=v(k)$, we have

\be\label{Wjv:dfr}\begin{aligned}\|\cW_Jv\|_{\dlrs{2}{1}{(sJ+r)}}^2=&\sum_{k\in\dZ}\|\la h^v_{\dm_0^J},\psi^0(\cdot-k)\ra_{\dLp{2}}\|^2+\sum_{j=0}^{J-1}\sum_{k\in\dZ}|\la h^v_{\dm_0^J},|\det(\dm_0^{-1}\dm_l)|^{\frac{1}{2}}\psi^l_{\dm_0^j;\dm_0^{-1}\dm_lk}\ra_{\dLp{2}}|^2\\
\le & D\|h^v\|_{\dLp{2}}^2,
\end{aligned}\ee
for some constant $D>0$. Hence \er{ub:v} and \er{Wjv:dfr} together yield
\be\label{WJ:ub}\|\cW_Jv\|_{\dlrs{2}{1}{(sJ+r)}}^2\le  DC_2^{-1}\|v\|_{\dlr{2}{r}}^2,\qquad v\in\dlr{2}{r},J\in\N.\ee
Similarly we can prove that there exists $D'>0$ such that
\be\label{tWJ:ub}\|\tilde{\cW}_Jv\|_{\dlrs{2}{1}{(sJ+r)}}^2\le  D'C_2^{-1}\|v\|_{\dlr{2}{r}}^2,\qquad v\in\dlr{2}{r},J\in\N.\ee
Thus by item (ii) of Theorem~\ref{stfr}, we see that both $\{b_l!\dm_l\}_{l=0}^s)$ and $\{\tilde{b}_l!\dm_l\}_{l=0}^s$ have stability in $\dlp{2}$. This completes the proof.

\ep

A similar result that connects discrete wavelet filter banks in $\dlp{2}$ and biorthogonal wavelets in $\dLp{2}$ can be established, which is the following theorem.

\begin{theorem}\label{charwvl2}Let $b_0,\tilde{b}_0\in\dlrs{0}{r}{r}, b_1\dots,b_s,\tilde{b}_1,\dots,\tilde{b}_s\in\dlrs{0}{1}{r}$ be finitely supported filters and let $\dm_0,\dots,\dm_s$ be $d\times d$ dilation matrices. Suppose $\psi^0,\tilde{\psi}^0\in\dLr{2}{r}$ are compactly supported standard refinable vector functions satisfying \er{ref:f} and $\ol{\wh{\psi^0}(0)}^{\tp}\wh{\tilde{\psi^0}(0)}=1$. Define $\psi^0,\dots,\psi^s,\tilde{\psi}^0,\tilde{\psi}^2\in\dLp{2}$ via \er{aff}. Then $(\{\psi^l!\dm_l\}_{l=0}^s,\{\tilde{\psi}^l!\dm_l\}_{l=0}^s)$ is a biorthogonal wavelet in $\dLp{2}$ if the following conditions are satisfied:

\begin{enumerate} 
	
	\item[(i)] $(\{b_l!\dm_l\}_{l=0}^s,\{\tilde{b}_l!\dm_l\}_{l=0}^s)$ is a biorthogonal wavelet filter bank with mixed sampling factors. 
	
	\item[(ii)] $\{b_l!\dm_l\}_{l=0}^s$ $\{\tilde{b}_l!\dm_l\}_{l=0}^s$ and  have stability in $\dlp{2}$.
	
	\item[(iii)]  The biorthogonality relation
	\be\label{bio:re}\la \tilde{\psi}^0,\psi^0(\cdot-k)\ra_{\dLp{2}}=\td(k)I_r\ee
	holds for all $k\in\dZ$.
	
\end{enumerate}
Conversely, if $(\{\psi^l!\dm_l\}_{l=0}^s,\{\tilde{\psi}^l!\dm_l\}_{l=0}^s)$ is a biorthogonal wavelet in $\dLp{2}$ and assume in addition that \er{llbs} holds for some constant $C'>0$, then items (i)-(iii) hold.

\end{theorem}

\bp By Theorem~\ref{charfrl2}, items (i) and (ii) imply that $(\{\psi^l!\dm_l\}_{l=0}^s,\{\tilde{\psi}^l!\dm_l\}_{l=0}^s)$ is a dual framelet in $\dLp{2}$. On the other hand, define $b_{l,j}, \tilde{b}_{l,j}$ as in \er{blj:1} and \er{tblj:1} and define $b_{l,j;k},\tilde{b}_{l,j;k}$ as in \er{b:l:j:k} and \er{tb:l:j:k} for all $l=0,\dots,s, j\in\N$ and $k\in\dZ$. It follows from \er{aff} that
\be\label{rel:J}\psi^l(x)=|\det(\dm_0)|^j\sum_{k\in\dZ}b_{l,j}(k)\psi^0(\dm_0^jx-k),\quad \tilde{\psi}^l(x)=|\det(\dm_0)|^j\sum_{k\in\dZ}\tilde{b}_{l,j}(k)\tilde{\psi}^0(\dm_0^jx-k)\ee
for a.e. $\xi\in\dR$ and for all $l=0,\dots,s$ and $j\in\N$. By calculation, we have
\be\label{biortho:psi}\begin{aligned}
&\la|\det(\dm_0^{-1}\dm_l)|^{\frac{1}{2}}\tilde{\psi}^l_{I_d;\dm_0^{-1}\dm_lm},|\det(\dm_0^{-1}\dm_t)|^{\frac{1}{2}}\psi^t_{\dm_0^{j-1};\dm_0^{-1}\dm_tn}\ra_{\dLp{2}}\\
=&|\det(\dm_0^{-2}\dm_l\dm_t)|^{\frac{1}{2}}|\det(\dm_0)|^{\frac{j-1}{2}}|\det(\dm_0)|^{j+1}\\
&\quad\times\sum_{p,q\in\dZ}\tilde{b}_{l,j}(p-\dm_0^{j-1}\dm_lm)\left(\int_{\dR}\tilde{\psi}^0(x-p)\ol{\psi^0(x-q)}^{\tp}dx\right)\ol{b_t(q-M_tn)}^{\tp}\\
=&|\det(\dm_0)|^{\frac{j-1}{2}}|\det(\dm_l)|^{\frac{1}{2}}|\det(\dm_t)|^{\frac{1}{2}}\\
&\quad\times\sum_{p,q\in\dZ}\tilde{b}_{l,j}(p-\dm_0^{j-1}\dm_lm)\left(\int_{\dR}\tilde{\psi}^0(x-p)\ol{\psi^0(x-q)}^{\tp}dx\right)\ol{b_t(q-\dm_tn)}^{\tp}\\
=&\sum_{p,q\in\dZ}\tilde{b}_{l,j;m}(p)\la \psi^0(\cdot-p),\tilde{\psi}^0(\cdot-q)\ra_{\dLp{2}}\ol{b_{t,1;n}(q)}^{\tp}\\
=&\sum_{p,q\in\dZ}\delta_0(p-q)\tilde{b}_{l,j;m}(p)\ol{b_{t,1;n}(q)}^{\tp}\\
=&\sum_{p\in\dZ}\tilde{b}_{l,j;m}(p)\ol{b_{t,1;n}(p)}^{\tp}\\
=&\la \tilde{b}_{l,j;m},b_{t,1;n}\ra_{\dlp{2}},
\end{aligned}\ee
for all $l,t=0,\dots,s$, $j\in\N$ and $m,n\in\dZ$. Similarly, by a simple scaling technique, one can prove that
\be\label{biortho:j}\begin{aligned}
&\la|\det(\dm_0^{-1}\dm_l)|^{\frac{1}{2}}\psi^l_{\dm_0^{j'};\dm_0^{-1}\dm_lm},|\det(\dm_0^{-1}\dm_l)|^{\frac{1}{2}}\tilde{\psi}^t_{\dm_0^{j-1};\dm_0^{-1}\dm_tn}\ra_{\dLp{2}}=&\la b_{l,j;m},\tilde{b}_{t,j';n}\ra_{\dlp{2}}
\end{aligned}\ee
for all $l,t=0,\dots,s$, $m,n\in\dZ$ and $j,j'\in\N$. By Theorem~\ref{wvdas}, we see that $(\AS(\{\psi^l!\dm_l\}_{l=0}^s),\AS(\{\tilde{\psi}^l!\dm_l\}_{l=0}^s))$ is a pair of biorthogonal sequences in $\dLp{2}$. Hence $(\{\psi^l!\dm_l\}_{l=0}^s,\{\tilde{\psi}^l!\dm_l\}_{l=0}^s)$ is a biorthogonal wavelet in $\dLp{2}$.\\

Conversely, suppose that $(\{\psi^l!\dm_l\}_{l=0}^s,\{\tilde{\psi}^l!\dm_l\}_{l=0}^s)$ is a biorthogonal wavelet in $\dLp{2}$. Then item (iii) trivially holds. \\


Next, by the frame property of $\AS\{\psi^l!\dm_l\}_{l=0}^{s}$ and $\AS\{\tilde{\psi}^l!\dm_l\}_{l=0}^{s}$, \cite[Proposition 4.4.13]{hanbook} yields that \er{ulbs} holds for some constant $C>0$. If in addition that \er{llbs} holds for some $C'>0$, then by the proof of Theorem~\ref{charfrl2}, we conclude that $\{b_l!\dm_l\}_{l=0}^s$ and $\{\tilde{b}_l!\dm_l\}_{l=0}^s$ have stability in $\dlp{2}$. This proves item (ii).\\

Finally, using \er{biortho:psi}, \er{biortho:j} and the biorthogonality of $(\AS\{\psi^l!\dm_l\}_{l=0}^{s},\AS\{\tilde{\psi}^l!\dm_l\}_{l=0}^{s})$, we see that $(\DAS\{b^l!\dm_l\}_{l=0}^{s},\DAS\{\tilde{b}^l!\dm_l\}_{l=0}^{s})$ is a pair of biothogonal sequences. Hence, by Theorem~\ref{wvdas}, we conclude that $(\{b_l!\dm_l\}_{l=0}^s,\{\tilde{b}_l!\dm_l\}_{l=0}^s)$ is a biorthogonal wavelet filter bank with mixed sampling factors. This proves item (i).

\ep

\end{document}